\documentclass[12pt,a4paper]{amsart}
\usepackage{amsmath,amsthm,amssymb}
\usepackage[truedimen,margin=25truemm]{geometry}
\usepackage{graphicx}
\usepackage{multicol,multirow}
\usepackage{wrapfig}

\usepackage{comment}
\usepackage[hang,small,bf]{caption}
\usepackage[subrefformat=parens]{subcaption}
\theoremstyle{plain}

\newtheorem{definition}{Definition}
\newtheorem{thm}[definition]{Theorem}

\newtheorem{lem}[definition]{Lemma}
\newtheorem{cor}[definition]{Corollary}
\newtheorem{rei}[definition]{Example}

\def\G{\Gamma}
\def\g{\gamma}

\def\Pe{L}
\DeclareMathOperator{\Ex}{{E}}
\DeclareMathOperator{\Va}{{V}}
\DeclareMathOperator{\Erf}{{erf}}
\DeclareMathOperator{\Sgn}{{sgn}}

\begin{document}
\title[]{Minimizing the expected value of the asymmetric loss and an inequality of the variance of the loss}
\author[]{Naoya Yamaguchi, Yuka Yamaguchi, and Ryuei Nishii}
\date{\today}

\maketitle

\begin{abstract}
For some estimations and predictions, 
we solve minimization problems with asymmetric loss functions. 
Usually, we estimate the coefficient of regression for these problems. 
In this paper, we do not make such the estimation, 
but rather give a solution by correcting any predictions so that the prediction error follows a general normal distribution. 
In our method, we can not only minimize the expected value of the asymmetric loss, 
but also lower the variance of the loss. 
\end{abstract}

\section{Introduction}\label{S1}
%It is important to make predictions. 
%For example, 
%The weather prediction is important in power generation, transportation, military, etc.
%However, 
For some estimations and predictions, 
we solve minimization problems with loss functions, as follows: 
Let $\{ (x_{i}, y_{i}) \mid 1 \leq i \leq  n \}$ be a data set, 
where $x_{i}$ are $1 \times p$ vectors and $y_{i} \in \mathbb{R}$. 
We assume that the data relate to a linear model, 
$$
y = X \beta + \varepsilon, 
$$
where $y = {}^{t}(y_{1}, \ldots, y_{n})$, $\varepsilon = {}^{t}(\varepsilon_{1}, \ldots, \varepsilon_{n})$, 
and $X$ is the $n \times p$ matrix having $x_{i}$ as the $i$th row. 
Let $L$ be a loss function and let $r_{i}(\beta) := y_{i} - x_{i} \beta$. 
Then we estimate the value: 
\begin{align*}
\hat{\beta} := \arg\min_{\beta} \left\{ \sum_{i = 1}^{n} L(r_{i}(\beta)) \right\}. 
\end{align*}
The case of $L(r_{i}(\beta)) = r_{i}(\beta)^{2}$ is well-known (see, e.g., Refs.~\cite{doi:10.1111/j.1751-5823.1998.tb00406.x}, \cite{legendre1805nouvelles}, and  \cite{stigler1981}). 
In the case of an asymmetric loss function, 
we refer the reader to, e.g., Refs.~\cite{10.2307/2336317}, \cite{10.2307/24303995}, \cite{10.2307/1913643}, and \cite{10.2307/2289234}. 
These studies estimate the parameter~$\hat{\beta}$. 
In this paper, however, we do not make such the estimation, 
but instead give a solution to the minimization problems by correcting any predictions so that the prediction error follows a general normal distribution. 
In our method, we can not only minimize the expected value of the asymmetric loss, 
but also lower the variance of the loss. 

Let $y$ be an observation value, and let $\hat{y}$ be a predicted value of $y$. 
We derive the optimized predicted value $y^{*} = \hat{y} + C$ minimizing the expected value of the loss under the assumption: 
\begin{enumerate}
\item 
The prediction error $z := \hat{y} - y$ is the realized value of a random variable $Z$, 
whose density function is a generalized Gaussian distribution function (see, e.g., Refs.~\cite{Dytso2018}, \cite{doi:10.1080/02664760500079464}, and \cite{Sub23}) with mean zero 
\begin{align*}
f_{Z}(z) := \frac{1}{2 a b \G(a)} \exp{\left( - \left\lvert \frac{z}{b} \right\rvert^{\frac{1}{a}} \right)}, 
\end{align*}
where $\G(a)$ is the gamma function and $a$, $b \in \mathbb{R}_{> 0}$. 
\item Let $k_{1}$, $k_{2} \in \mathbb{R}_{> 0}$. 
If there is a mismatch between $y$ and $\hat{y}$, then we suffer a loss, 
\begin{align*}
\Pe(z) := 
\begin{cases}
k_{1} z, & z \geq 0, \\ 
- k_{2} z, & z < 0. 
\end{cases}
\end{align*}
\end{enumerate}
That is, the solution to the minimization problem is 
\begin{align*}
C = \arg\min_{c} \left\{ \Ex\left[ \Pe(Z + c) \right] \right\}. 
\end{align*}

The motivation of our research is as follows: 
(1) Predictions usually cause prediction errors. 
Therefore, it is necessary to use predictions in consideration of predictions errors. 
Actually, in some cases, it is best not to act as predicted because of prediction errors.
For example, 
the paper~\cite{Yamaguchi2018} formulates a method for minimizing the expected value of the procurement cost of electricity in two popular spot markets: {\it day-ahead} and {\it intra-day}, 
under the assumption that the expected value of the unit prices and the distributions of the prediction errors for the electricity demand traded in two markets are known. 
The paper showed that if the procurement is increased or decreased from the prediction, 
in some cases, the expected value of the procurement cost is reduced.
(2) In recent years, prediction methods have been black boxed by the big data and machine learning (see, e.g., Ref.~\cite{10.1145/3236009}). 
The day will soon come, when we must minimize the objective function by using predictions obtained by such black boxed methods.
In our method, 
even if we do not know the prediction $\hat{y}$, 
we can determine the parameter $C$ if we know the prediction error distribution $f$ and asymmetric loss function $L$.

To obtain $y^{*}$, we derive $\Ex[\Pe(Z + c)]$ for any $c \in \mathbb{R}$. 
Let $\G(a, x)$ and  $\g(a, x)$ be the upper and the lower incomplete gamma functions, respectively (see, e.g., Ref.~\cite{doi:10.1142/0653}). 
The expected value and the variance of $\Pe(Z + c)$ are as follows: 
\begin{lem}\label{lem:1.1}
For any $c \in \mathbb{R}$, we have 
\begin{align*}
(1)\quad 
\Ex[\Pe(Z + c)] 
&= \frac{(k_{1} - k_{2}) c}{2} 
+ \frac{(k_{1} + k_{2}) \lvert c \rvert}{2 \G(a)} 
\g\left(a, \left\lvert \frac{c}{b} \right\rvert^{\frac{1}{a}} \right) 
+ \frac{(k_{1} + k_{2}) b}{2 \G(a)} 
\G\left(2a, \left\lvert \frac{c}{b} \right\rvert^{\frac{1}{a}} \right), \\
(2)\quad 
\Va[\Pe(Z + c)] 
&= \frac{(k_{1} + k_{2})^{2} c^{2}}{4} 
+ \frac{(k_{1}^{2} - k_{2}^{2}) b c}{2 \G(a)} 
\G\left(2a, \left\lvert \frac{c}{b} \right\rvert^{\frac{1}{a}} \right) \nonumber \\
&\quad - \frac{(k_{1} + k_{2})^{2} b \lvert c \rvert}{2 \G(a)^{2}}
\g\left(a, \left\lvert \frac{c}{b} \right\rvert^{\frac{1}{a}} \right)
\G\left(2a, \left\lvert \frac{c}{b} \right\rvert^{\frac{1}{a}} \right) \nonumber \\
&\quad - \frac{(k_{1} + k_{2})^{2} c^{2}}{4 \G(a)^{2}}
\g\left(a, \left\lvert \frac{c}{b} \right\rvert^{\frac{1}{a}} \right)^{2} 
- \frac{(k_{1} + k_{2})^{2} b^{2}}{4 \G(a)^{2}}
\G\left(2a, \left\lvert \frac{c}{b} \right\rvert^{\frac{1}{a}} \right)^{2} \nonumber \\
&\quad + \frac{(k_{1}^{2} + k_{2}^{2}) b^{2} \G(3a)}{2 \G(a)}  
+ \Sgn(c) \frac{(k_{1}^{2} - k_{2}^{2}) b^{2}}{2 \G(a)} 
\g\left(3a, \left\lvert \frac{c}{b} \right\rvert^{\frac{1}{a}} \right). \nonumber
\end{align*}
\end{lem}

We write the value of $c$ satisfying $\frac{d}{dc} \Ex[\Pe(Z + c)] = 0$ as $C$. 
Then, we find that $\Ex[\Pe(Z + c)]$ has a minimum value at $c = C$. 
Also, it follows from  
\begin{align}
\g\left(a, \left\lvert \frac{C}{b} \right\rvert^{\frac{1}{a}} \right) 
= \Sgn(C) \frac{k_{2} - k_{1}}{k_{1} + k_{2}} \G(a) 
\end{align}
that $\Sgn(C) = \Sgn(k_{2} - k_{1})$, 
where $\Sgn(c) := 1 \: (c \geq 0); -1 \: (c < 0)$, 
and $C = 0$ only when $k_{1} = k_{2}$. 
This equation implies that the ratio of $\G(a)$ and 
$\g\left(a, \left\lvert \frac{C}{b} \right\rvert^{\frac{1}{a}} \right)$ is 
$1 : \frac{\lvert k_{2} - k_{1}\rvert}{k_{1} + k_{2}}$. 
That is, the vertical axis $t = \left\lvert \frac{C}{b} \right\rvert^{\frac{1}{a}}$ divides the area between $t^{a - 1} e^{- t}$ and the $t$-axis into $\frac{\lvert k_{2} - k_{1}\rvert}{k_{1} + k_{2}} : 1- \frac{\lvert k_{2} - k_{1}\rvert}{k_{1} + k_{2}}$. 

Substituting $c = C$ in the equation $(1)$ of Lemma~$\ref{lem:1.1}$, 
from the equation~$(3)$, 
we have 
\begin{align*}
\Ex[\Pe(Z + C)] 
= \frac{(k_{1} + k_{2}) b}{2 \G(a)} 
\G\left(2a, \left\lvert \frac{C}{b} \right\rvert^{\frac{1}{a}} \right). 
\end{align*}
This is the minimum value of $\Ex[\Pe(Z + c)]$. 
From this and the $c = 0$ case of the equation $(1)$ of Lemma~$\ref{lem:1.1}$, we have the following corollary: 

\begin{cor}\label{cor:1.2}
We have 
\begin{align*}
\Ex[\Pe(Z)] - \Ex[\Pe(Z + C)] 
&= \frac{(k_{1} + k_{2}) b}{2 \G(a)} 
\g\left(2a, \left\lvert \frac{C}{b} \right\rvert^{\frac{1}{a}} \right), \\ 
\frac{\Ex[\Pe(Z + C)]}{\Ex[\Pe(Z)]} 
&= \frac{1}{\G(2a)} \G\left(2a, \left\lvert \frac{C}{b} \right\rvert^{\frac{1}{a}} \right). 
\end{align*}
\end{cor}
This corollary asserts that the expected value of the loss is reduced by correcting a predicted value $y$ to the optimized predicted value $y^{*}$. 
Moreover, the following holds: 

\begin{thm}\label{thm:1.3}
We have 
\begin{align*}
 \Va[\Pe(Z + C)] \leq \Va[\Pe(Z)], 
\end{align*}
where equality sign holds only when $C = 0$; that is, when $k_{1} = k_{2}$.  
\end{thm}

This theorem asserts that the variance of the loss is reduced by correcting the predicted value $y$ to the optimized predicted value $y^{*}$. 
To prove this theorem, we use the following lemma: 
\begin{lem}\label{lem:1.4}
For $a > 0$ and $x > 0$, 
we have 
\begin{align*}
x^{a} \g(a, x)^{2} - x^{a} \G(a)^{2} + 2 \g(a, x) \G(2a, x) > 0. 
\end{align*}
\end{lem}

To prove Lemma~$\ref{lem:1.4}$, we use the following lemmas: 
\begin{lem}\label{lem:1.5}
For $a > 0$, 
we have 
\begin{align*}
2 \G(2a) - a \G(a)^{2} > 0. 
\end{align*}
\end{lem}

\begin{lem}\label{lem:1.6}
For $a > 0$, 
we have 
\begin{align*}
4^{a} \G\left(a + \frac{1}{2} \right) > \sqrt{\pi} \G(a + 1). 
\end{align*}
\end{lem}

The remainder of this paper is organized as follows. 
In Section~$2$, we set up the problem. 
In Section~$3$, we introduce the expected value and the variance of $\Pe(Z + c)$, 
and we determine the value of $c = C$ that gives the minimum value of $\Ex[\Pe(Z + c)]$. 
In addition, we give a geometrical interpretation of the parameter $C$, 
and give the minimized expected value $\Ex[\Pe(Z + C)]$. 
In Section~$4$, we prove Theorem~$\ref{thm:1.3}$. 
In Section~$5$, we give some inequalities for the gamma and the incomplete gamma functions, 
which used to derive the inequality for the variance of the loss in Theorem~$\ref{thm:1.3}$. 
In Section~$6$, we write the calculation of the expected value and the variance of the loss $\Pe(Z + c)$ for $c \in \mathbb{R}$.

\section{Problem statement}
In this section, we set a problem. 
Let $y$ be an observation value, 
let $\hat{y}$ be a predicted value of $y$, 
and let $\G(a)$ be the gamma function (see, e.g., Ref.~\cite[p.~93]{doi:10.1142/0653}) defined by 
\begin{align*}
\G(a) := \int_{0}^{+\infty} t^{a - 1} e^{- t} dt, \quad \text{Re}(a) > 0. 
\end{align*}
We assume the following: 
\begin{enumerate}
\item 
The prediction error $z := \hat{y} - y$ is the realized value of a random variable $Z$, 
whose density function is a generalized Gaussian distribution function (see, e.g., Refs.~\cite{Dytso2018}, \cite{doi:10.1080/02664760500079464}, and \cite{Sub23}) with mean zero 
\begin{align*}
f_{Z}(z) := \frac{1}{2 a b \G(a)} \exp{\left( - \left\lvert \frac{z}{b} \right\rvert^{\frac{1}{a}} \right)}, 
\end{align*}
where $a$, $b \in \mathbb{R}_{> 0}$. 
\item Let $k_{1}$, $k_{2} \in \mathbb{R}_{> 0}$. 
If there is a mismatch between $y$ and $\hat{y}$, then we suffer a loss, 
\begin{align*}
\Pe(z) := 
\begin{cases}
k_{1} z, & z \geq 0, \\ 
- k_{2} z, & z < 0. 
\end{cases}
\end{align*}
\end{enumerate}
We derive the optimized predicted value $y^{*} = \hat{y} + C$ minimizing $\Ex[\Pe(Z + c)]$. 
For this purpose, we derive $\Ex[\Pe(Z + c)]$ for any $c \in \mathbb{R}$ in the next section.

\section{Expected value and variance of the loss}
Here, we introduce the expected value and the variance of $\Pe(Z + c)$, 
and determine the value of $c = C$ that gives the minimum value of $\Ex[\Pe(Z + c)]$. 
In addition, we give a geometrical interpretation of the parameter $C$ and give the minimized expected value $\Ex[\Pe(Z + C)]$. 

\subsection{Expected value and variance of the loss}
%We introduce the expected value and the variance of $\Pe(Z + c)$. 
Let $\G(a, x)$ and  $\g(a, x)$ be the upper and the lower incomplete gamma functions, respectively, defined by 
\begin{align*}
\G(a, x) := \int_{x}^{+\infty} t^{a - 1} e^{-t} dt, \qquad 
\g(a, x) := \int_{0}^{x} t^{a - 1} e^{-t} dt, 
\end{align*}
where $\text{Re}(a) > 0$ and $x \geq 0$. 
These functions have the following properties: 
\begin{lem}\label{lem:3.0}
For ${\rm Re}(a) > 0$ and $x \geq 0$, 
\begin{flalign*}\quad
&(1)\quad \g(a, x) + \G(a, x) = \G(a); \\
&(2)\quad \lim_{x \to \infty} \g(a, x) = \G(a); \\
&(3)\quad \G(a, 0) = \G(a); \\
&(4)\quad \frac{d}{dx} \g(a, x) = x^{a - 1} e^{-x}; \\
&(5)\quad \frac{d}{dx} \G(a, x) = - x^{a - 1} e^{-x}. 
&
\end{flalign*}
\end{lem}

Also, for $c \in \mathbb{R}$, 
let $\Sgn(c) := 1 \: (c \geq 0); -1 \: (c < 0)$. 
Then, the expected value and the variance of $\Pe(Z + c)$ are as follows: 
\begin{lem}[Section~$\ref{S1}$, Lemma~$\ref{lem:1.1}$]\label{lem:3.1}
For any $c \in \mathbb{R}$, we have 
\begin{align*}
(1)\quad 
\Ex[\Pe(Z + c)] 
&= \frac{(k_{1} - k_{2}) c}{2} 
+ \frac{(k_{1} + k_{2}) \lvert c \rvert}{2 \G(a)} 
\g\left(a, \left\lvert \frac{c}{b} \right\rvert^{\frac{1}{a}} \right) 
+ \frac{(k_{1} + k_{2}) b}{2 \G(a)} 
\G\left(2a, \left\lvert \frac{c}{b} \right\rvert^{\frac{1}{a}} \right), \\
(2)\quad 
\Va[\Pe(Z + c)] 
&= \frac{(k_{1} + k_{2})^{2} c^{2}}{4} 
+ \frac{(k_{1}^{2} - k_{2}^{2}) b c}{2 \G(a)} 
\G\left(2a, \left\lvert \frac{c}{b} \right\rvert^{\frac{1}{a}} \right) \nonumber \\
&\quad - \frac{(k_{1} + k_{2})^{2} b \lvert c \rvert}{2 \G(a)^{2}}
\g\left(a, \left\lvert \frac{c}{b} \right\rvert^{\frac{1}{a}} \right)
\G\left(2a, \left\lvert \frac{c}{b} \right\rvert^{\frac{1}{a}} \right) \nonumber \\
&\quad - \frac{(k_{1} + k_{2})^{2} c^{2}}{4 \G(a)^{2}}
\g\left(a, \left\lvert \frac{c}{b} \right\rvert^{\frac{1}{a}} \right)^{2} 
- \frac{(k_{1} + k_{2})^{2} b^{2}}{4 \G(a)^{2}}
\G\left(2a, \left\lvert \frac{c}{b} \right\rvert^{\frac{1}{a}} \right)^{2} \nonumber \\
&\quad + \frac{(k_{1}^{2} + k_{2}^{2}) b^{2} \G(3a)}{2 \G(a)}  
+ \Sgn(c) \frac{(k_{1}^{2} - k_{2}^{2}) b^{2}}{2 \G(a)} 
\g\left(3a, \left\lvert \frac{c}{b} \right\rvert^{\frac{1}{a}} \right). \nonumber
\end{align*}
\end{lem}

See the last two sections for the proof of Lemma~$\ref{lem:3.1}$.

From Lemma~$\ref{lem:3.1}$, we have the following: 
\begin{align}
\Ex[\Pe(Z)] 
&= \frac{(k_{1} + k_{2}) b}{2 \G(a)} 
\G\left( 2a \right),  \label{E[L(Z)]} \\ 
\Va[\Pe(Z)] 
&= \frac{(k_{1}^{2} + k_{2}^{2}) b^{2} \G(3a)}{2 \G(a)} 
- \frac{(k_{1} + k_{2})^{2} b^{2} \G(2a)^{2}}{4 \G(a)^{2}}. \label{V[L(Z)]}
\end{align}

Let $\Erf(x)$ be the error function defined by  
\begin{align*}
\Erf(x) := \frac{2}{\sqrt{\pi}} \int_{0}^{x} \exp{\left( - t^{2} \right)} dt 
\end{align*}
for any $x \in \mathbb{R}$. 
We give two examples of $\Ex[\Pe(Z + c)]$ and $\Va[\Pe(Z + c)]$. 
\begin{rei}\label{rei:3.2}
In the case of ${\rm Laplace}(0, b)$, since $a = 1$, we have 
\begin{align*}
\Ex[\Pe(Z + c)] 
&= \Pe(c) + \frac{(k_{1} + k_{2}) b}{2} \exp{\left(- \left\lvert \frac{c}{b} \right\rvert \right)}, \\ 
\Va[\Pe(Z + c)]
&= \left\{ k_{1}^{2} + k_{2}^{2} + \Sgn(c) (k_{1}^{2} - k_{2}^{2} ) \right\} b^{2} \\ 
&\quad - \Sgn(c) (k_{1} + k_{2}) \left\{ \Pe(c) + b (k_{1} - k_{2}) \right\} b \exp{\left(- \left\lvert \frac{c}{b} \right\rvert \right)} \\ 
&\quad - \frac{(k_{1} + k_{2})^{2} b^{2}}{4} \exp{\left(- 2 \left\lvert \frac{c}{b} \right\rvert \right)}. 
\end{align*}
In the case of $\mathcal{N}(0, \frac{1}{2} b^{2})$, 
since $a = \frac{1}{2}$, we have
\begin{align*}
\Ex[\Pe(Z + c)] 
&= \frac{(k_{1} - k_{2}) c}{2}  
+ \frac{(k_{1} + k_{2}) c}{2} \Erf{\left(\frac{c}{b} \right)} 
+ \frac{(k_{1} + k_{2}) b}{2 \sqrt{\pi}} \exp{\left(- \frac{c^{2}}{b^{2}} \right)}, \\ 
\Va[\Pe(Z + c)] 
&= \frac{(k_{1}^{2} + k_{2}^{2}) b^{2} }{4} 
+ \frac{(k_{1} + k_{2})^{2} c^{2} }{4} 
+ \frac{(k_{1}^{2} - k_{2}^{2}) b^{2} }{4} \Erf{\left( \frac{c}{b} \right)} 
- \frac{(k_{1} + k_{2})^{2} c^{2} }{4} \Erf^{2}{\left( \frac{c}{b} \right)} \\ 
&\quad - \frac{(k_{1} + k_{2})^{2} b c}{2 \sqrt{\pi}} \Erf{\left(\frac{c}{b} \right)} \exp{\left( - \frac{c^{2}}{b^{2}} \right)} 
- \frac{(k_{1} + k_{2})^{2} b^{2} }{4 \pi} \exp{\left( - \frac{2 c^{2}}{b^{2}} \right)}.  
\end{align*}
With the conditions fixed as $k_{1} = 50$ and $b = 1$, 
we can plot $\Ex[\Pe(Z)]$ and $\Va[\Pe(Z)]$ for the Laplace and the Gauss distributions as follows: 
\begin{figure}[h!]
\begin{tabular}{cc}
\begin{minipage}{0.5\hsize}
\begin{center}
\includegraphics[height=4cm]{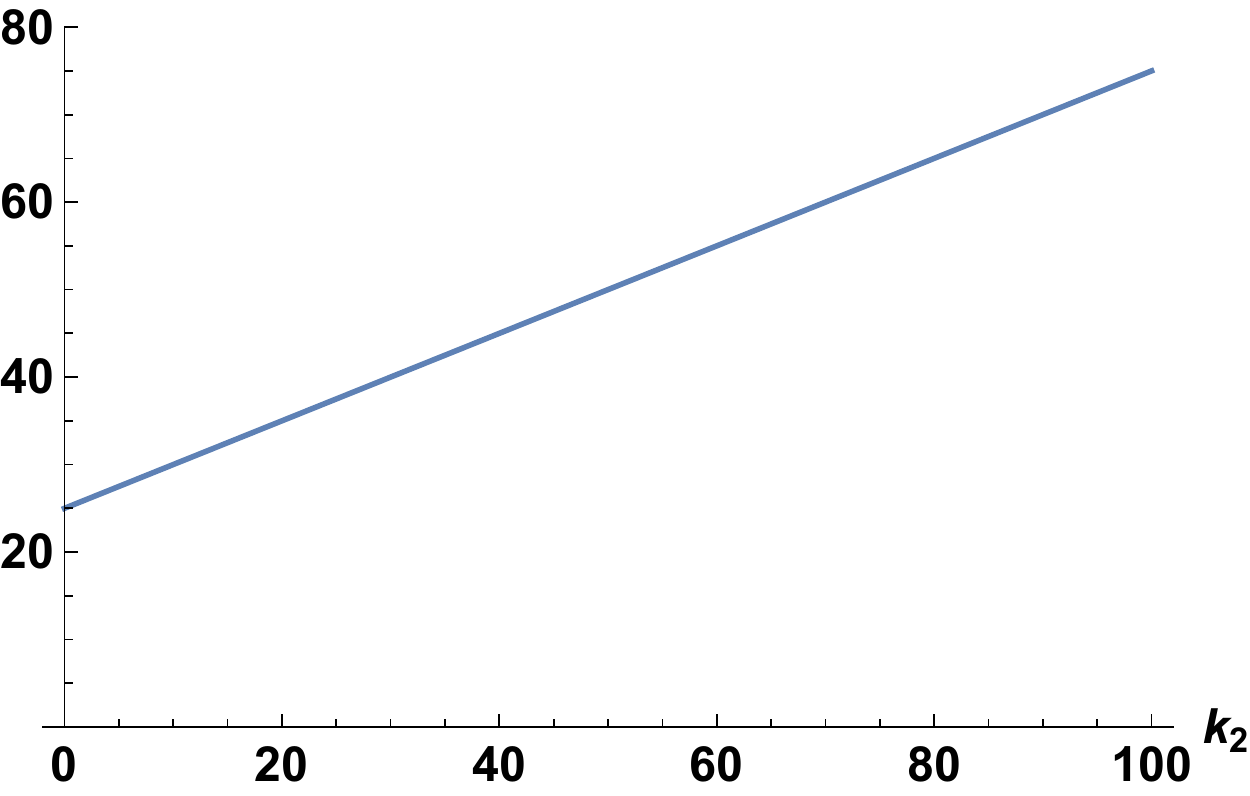}
\subcaption{$\Ex[\Pe(Z)]$}
\label{fig:winter}
\end{center}
\end{minipage}
\begin{minipage}{0.5\hsize}
\begin{center}
\includegraphics[height=4cm]{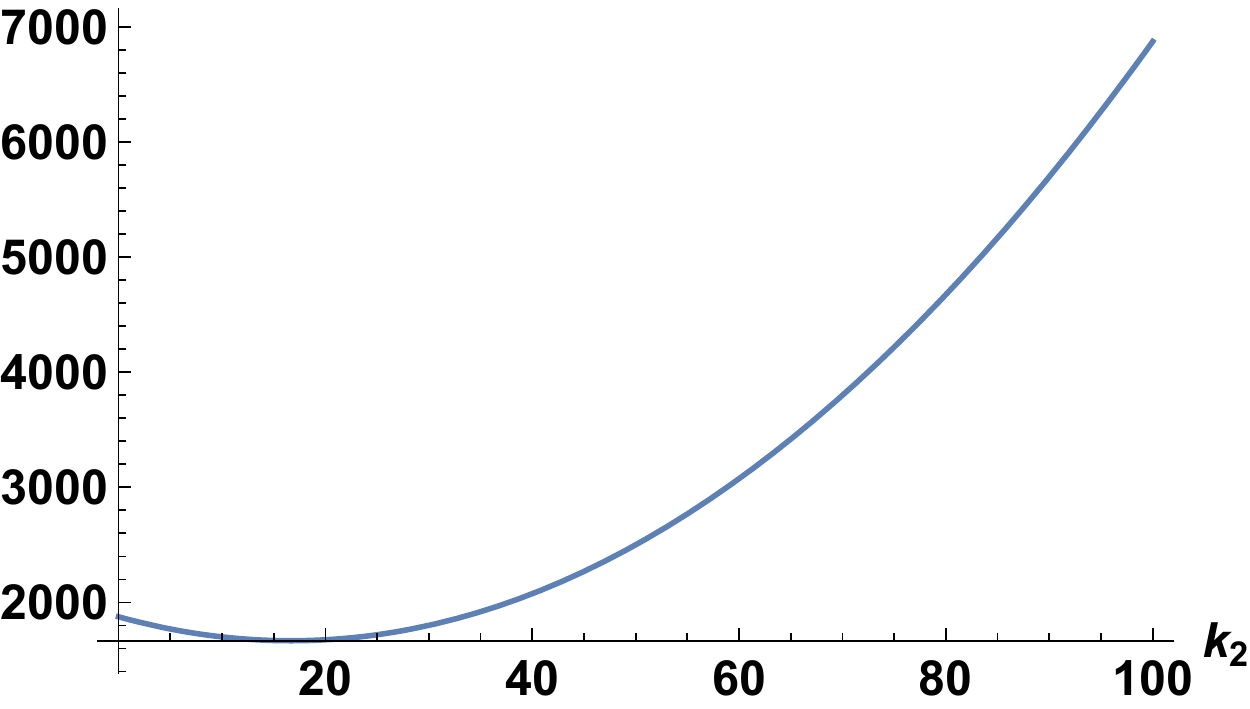}
\subcaption{$\Va[\Pe(Z)]$}
\label{fig:fall}
\end{center}
\end{minipage}
\end{tabular}
\caption{Plots for the Laplace distribution}
\end{figure}
\begin{figure}[h!]
\begin{tabular}{cc}
\begin{minipage}{0.5\hsize}
\begin{center}
\includegraphics[height=4cm]{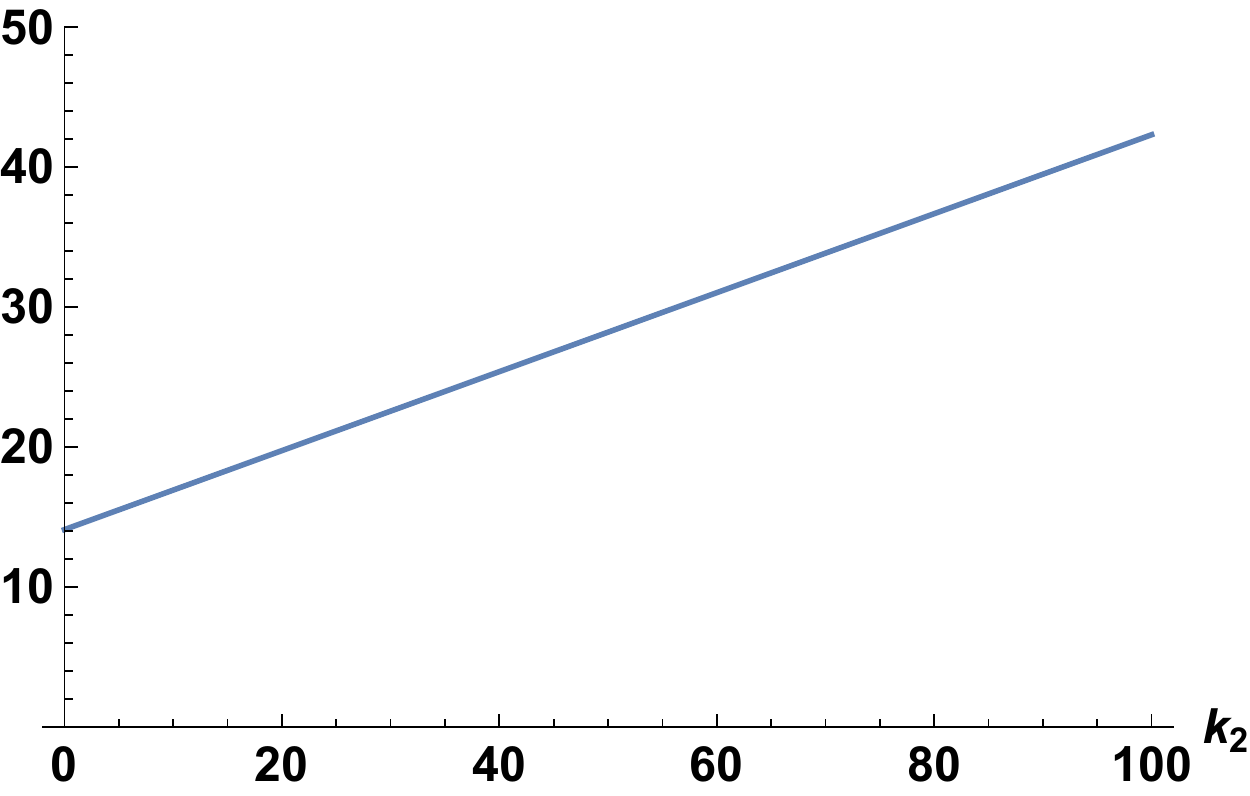}
\subcaption{$\Ex[\Pe(Z)]$}
\label{fig:winter}
\end{center}
\end{minipage}
\begin{minipage}{0.5\hsize}
\begin{center}
\includegraphics[height=4cm]{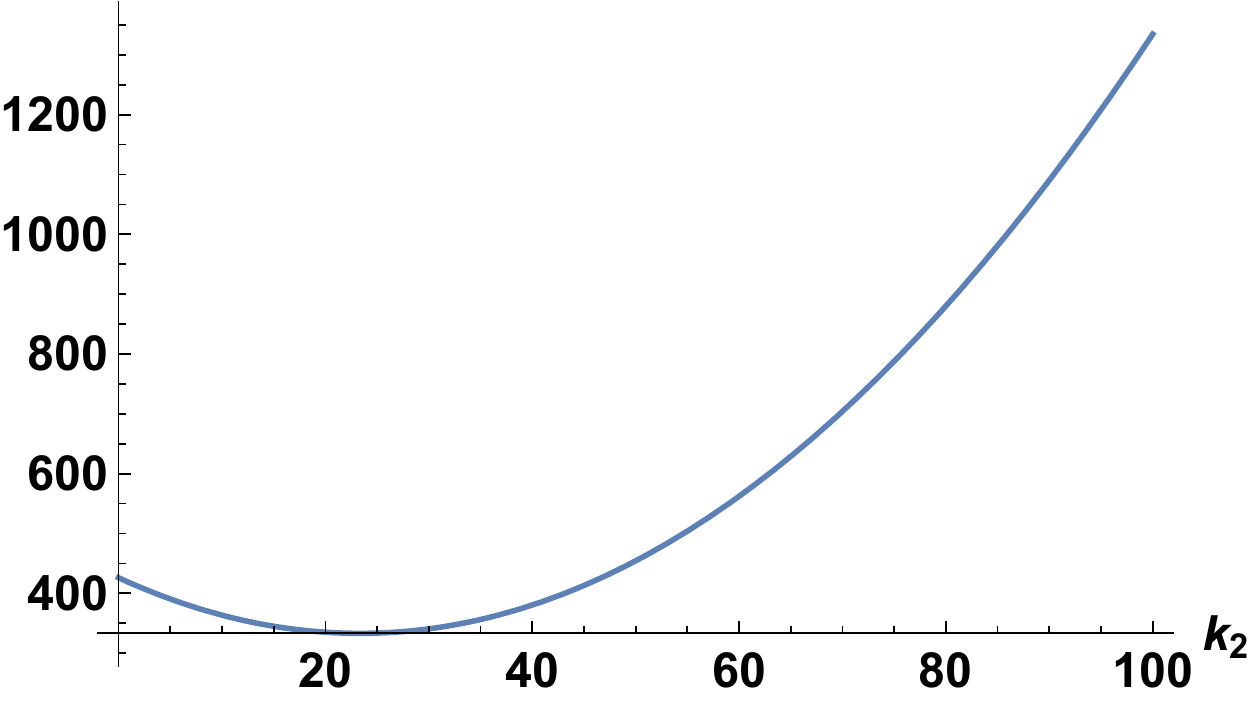}
\subcaption{$\Va[\Pe(Z)]$}
\label{fig:fall}
\end{center}
\end{minipage}
\end{tabular}
\caption{Plots for the Gauss distribution}
\end{figure}

\end{rei}

\subsection{Parameter value minimizing the expected value}
Here, we determine the value of $c = C$ that gives the minimum value of $\Ex[\Pe(Z + c)]$. 
Since 
\begin{align*}
\frac{d}{dc} \G\left(2a, \left\lvert \frac{c}{b} \right\rvert^{\frac{1}{a}} \right) 
&= - \frac{c}{a b} \exp{\left(- \left\lvert \frac{c}{b} \right\rvert^{\frac{1}{a}} \right)}; \\
\frac{d}{dc} \g\left(a, \left\lvert \frac{c}{b} \right\rvert^{\frac{1}{a}} \right) 
&= \Sgn(c) \frac{1}{a b} \exp{\left(- \left\lvert \frac{c}{b} \right\rvert^{\frac{1}{a}} \right)}, 
\end{align*}
we have 
\begin{align*}
\frac{d}{dc} \Ex[\Pe(Z + c)] = \frac{k_{1} - k_{2}}{2} 
+ \Sgn(c) \frac{k_{1} + k_{2}}{2 \G(a)} 
\g\left(a, \left\lvert \frac{c}{b} \right\rvert^{\frac{1}{a}} \right). 
\end{align*}
We will denote the value of $c$ satisfying $\frac{d}{dc} \Ex[\Pe(Z + c)] = 0$ as $C$. 
Then, from the first derivative test, 
we find that $\Ex[\Pe(Z + c)]$ has a minimum value at $c = C$. 
\begin{center}
\begin{tabular}{|c|c|c|c|}
\hline
$c$ & Less than $C$ & $C$ & More than $C$ \\ \hline 
$\frac{d}{dc} \Ex[\Pe(Z + c)]$ & Negative & $0$ & Positive \\ \hline 
$\Ex[\Pe(Z + c)]$ & Strongly decreasing & & Strongly increasing \\ \hline 
\end{tabular}
\end{center}
Also, it follows from  
\begin{align}\label{C}
\g\left(a, \left\lvert \frac{C}{b} \right\rvert^{\frac{1}{a}} \right) 
= \Sgn(C) \frac{k_{2} - k_{1}}{k_{1} + k_{2}} \G(a) 
\end{align}
that $\Sgn(C) = \Sgn(k_{2} - k_{1})$ and $C = 0$ only when $k_{1} = k_{2}$. 

Moreover, equation~$(\ref{C})$ implies that the ratio of $\G(a)$ and 
$\g\left(a, \left\lvert \frac{C}{b} \right\rvert^{\frac{1}{a}} \right)$ is $1 : \frac{\lvert k_{2} - k_{1}\rvert}{k_{1} + k_{2}}$. 
That is, the vertical axis $t = \left\lvert \frac{C}{b} \right\rvert^{\frac{1}{a}}$ divides the area between $t^{a - 1} e^{- t}$, and the $t$-axis into $\frac{\lvert k_{2} - k_{1}\rvert}{k_{1} + k_{2}} : 1- \frac{\lvert k_{2} - k_{1}\rvert}{k_{1} + k_{2}}$. 
\begin{figure}[h!]
 \begin{center}
  \includegraphics[width=80mm]{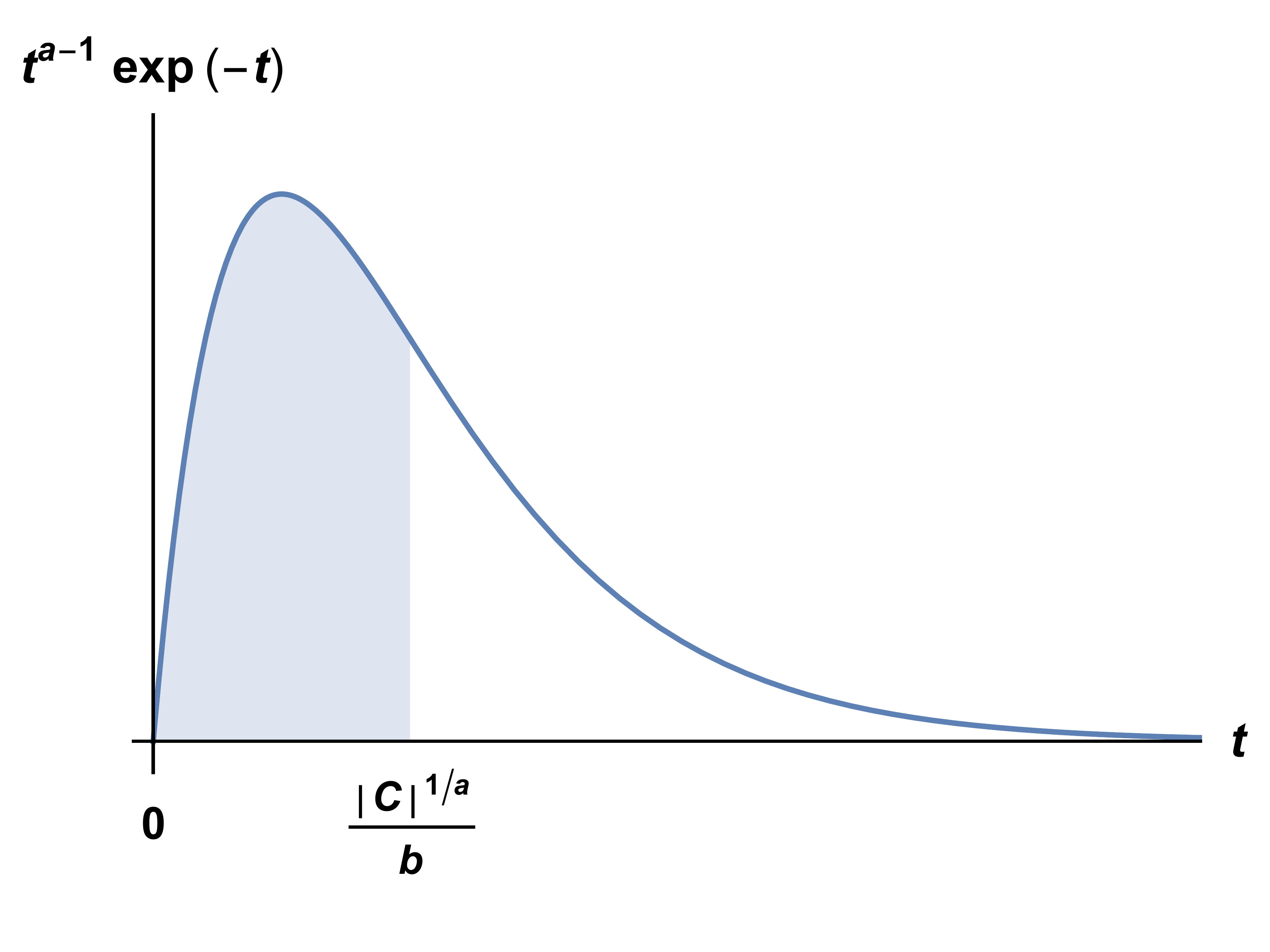}
  \caption{Plot of area ratio}
 \end{center}
\end{figure}

Let $\Erf^{-1}(x)$ be the inverse error function. 
We give two examples of $C$. 
\begin{rei}\label{rei:}
In the case of ${\rm Laplace}(0, b)$, since $a = 1$, we have 
\begin{align*}
C = - \Sgn(k_{2} - k_{1}) b \log{\left( 1 - \left\lvert \frac{k_{2} - k_{1}}{k_{1} + k_{2}} \right\rvert \right)}. 
\end{align*}
In the case of $\mathcal{N}(0, \frac{1}{2} b^{2})$, 
since $a = \frac{1}{2}$, we have
\begin{align*}
C = \Sgn(k_{2} - k_{1}) b \Erf^{-1}\left( \left\lvert \frac{k_{2} - k_{1}}{k_{1} + k_{2}} \right\rvert \right). 
\end{align*}
Fixing the conditions as $k_{1} = 50$ and $b = 1$, 
we can plot $C$ for the Laplace and the Gauss distributions as follows: 
\begin{figure}[h!]
\begin{tabular}{cc}
\begin{minipage}{0.5\hsize}
\begin{center}
\includegraphics[height=4cm]{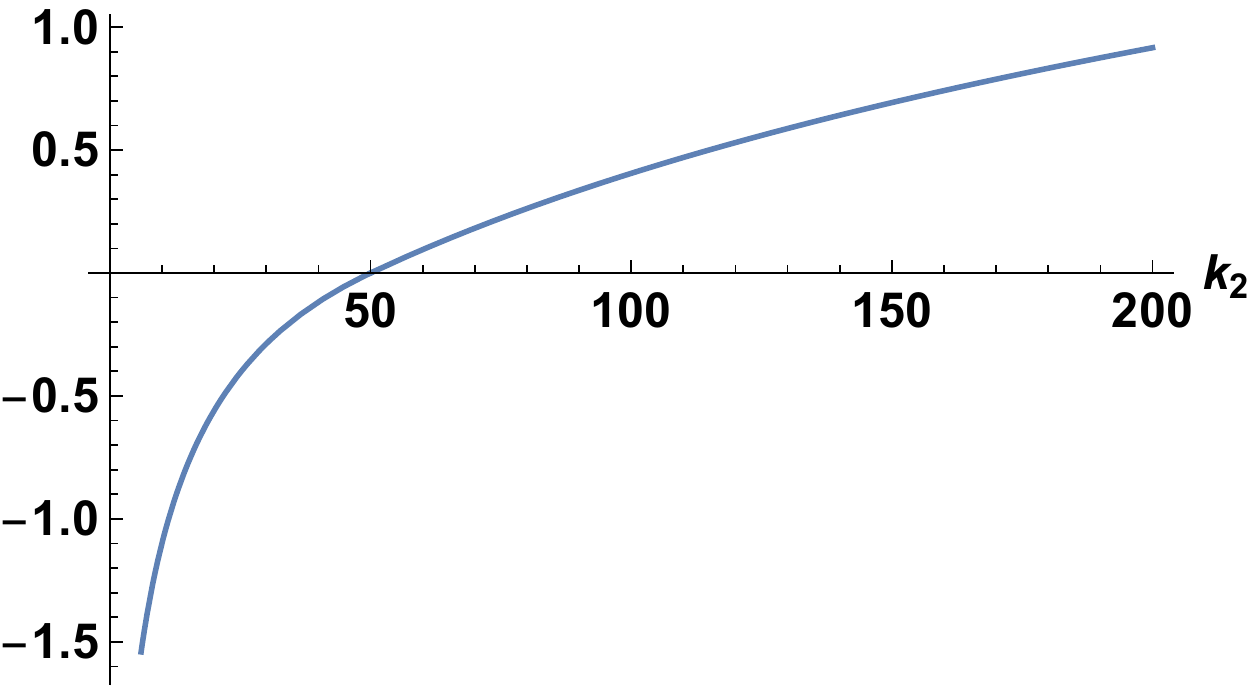}
\subcaption{When $Z \sim {\rm Laplace}(0, 1)$}
\label{fig:winter}
\end{center}
\end{minipage}
\begin{minipage}{0.5\hsize}
\begin{center}
\includegraphics[height=4cm]{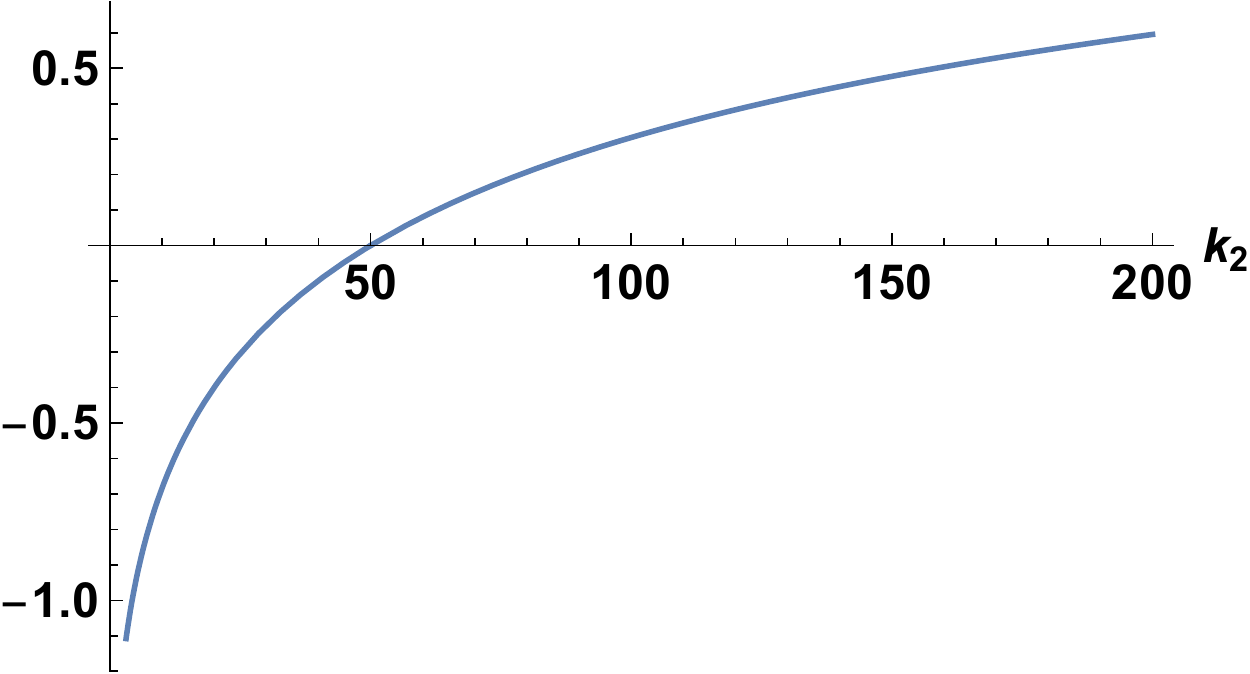}
\subcaption{When $Z \sim \mathcal{N}(0, \frac{1}{2})$}
\label{fig:fall}
\end{center}
\end{minipage}
\end{tabular}
\caption{Plots of $C$ for the Laplace and the Gauss distributions}
\end{figure}

\end{rei}

\subsection{Minimized expected value of the loss}
We give the minimum value of $\Ex[\Pe(Z + c)]$. 
Substituting $c = C$ in equation $(1)$ of Lemma~$\ref{lem:3.1}$, 
from equation~$(\ref{C})$, 
we have 
\begin{align*}
\Ex[\Pe(Z + C)] 
= \frac{(k_{1} + k_{2}) b}{2 \G(a)} 
\G\left(2a, \left\lvert \frac{C}{b} \right\rvert^{\frac{1}{a}} \right). 
\end{align*}
This is the minimum value of $\Ex[\Pe(Z + c)]$. 
From this and equation $(\ref{E[L(Z)]})$, we have the following corollary: 

\begin{cor}[Section~$\ref{S1}$, Corollary~$\ref{cor:1.2}$]\label{cor:3.3}
We have 
\begin{align*}
\Ex[\Pe(Z)] - \Ex[\Pe(Z + C)] 
&= \frac{(k_{1} + k_{2}) b}{2 \G(a)} 
\g\left(2a, \left\lvert \frac{C}{b} \right\rvert^{\frac{1}{a}} \right), \\ 
\frac{\Ex[\Pe(Z + C)]}{\Ex[\Pe(Z)]} 
&= \frac{1}{\G(2a)} \G\left(2a, \left\lvert \frac{C}{b} \right\rvert^{\frac{1}{a}} \right). 
\end{align*}
\end{cor}

Fixing the conditions as $k_{1} = 50$ and $b = 1$, 
we can plot the plots of $\Ex[\Pe(Z)] - \Ex[\Pe(Z + C)]$ for the Laplace and the Gauss distributions as follows: 
\begin{figure}[h!]
\begin{tabular}{cc}
\begin{minipage}{0.5\hsize}
\begin{center}
\includegraphics[height=4cm]{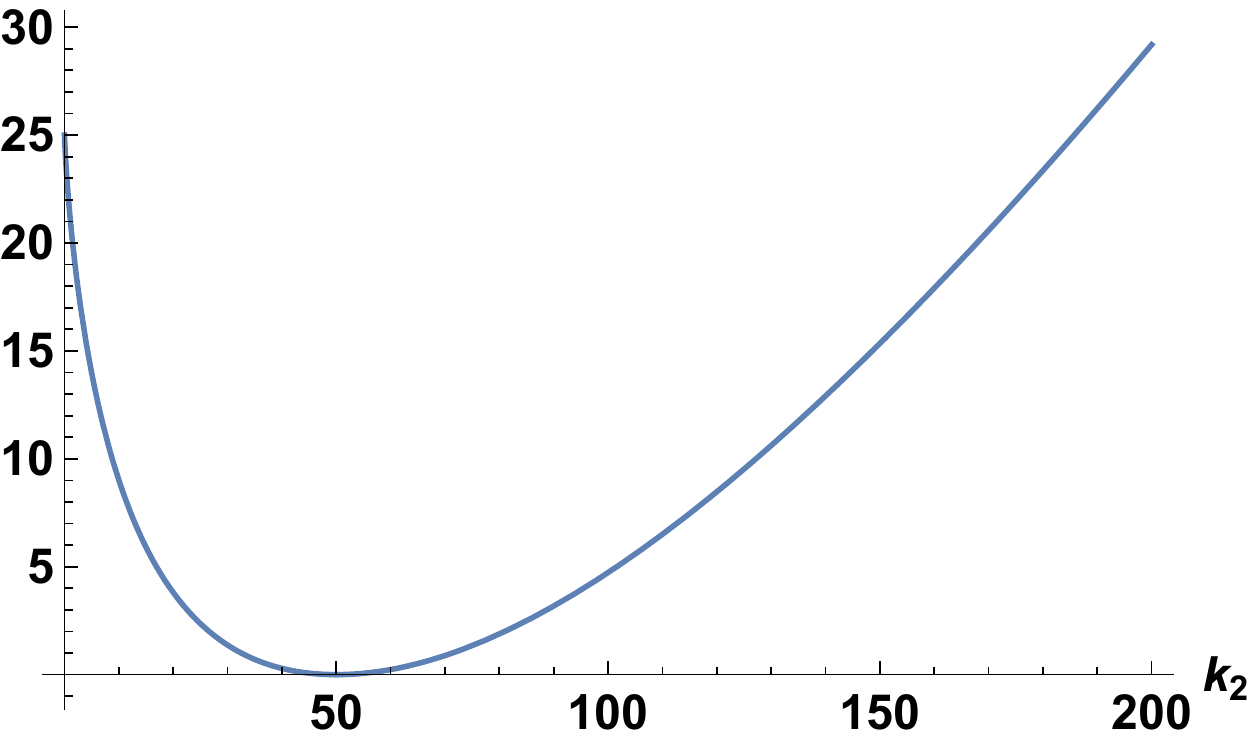}
\subcaption{When $Z \sim {\rm Laplace}(0, 1)$}
\label{fig:winter}
\end{center}
\end{minipage}
\begin{minipage}{0.5\hsize}
\begin{center}
\includegraphics[height=4cm]{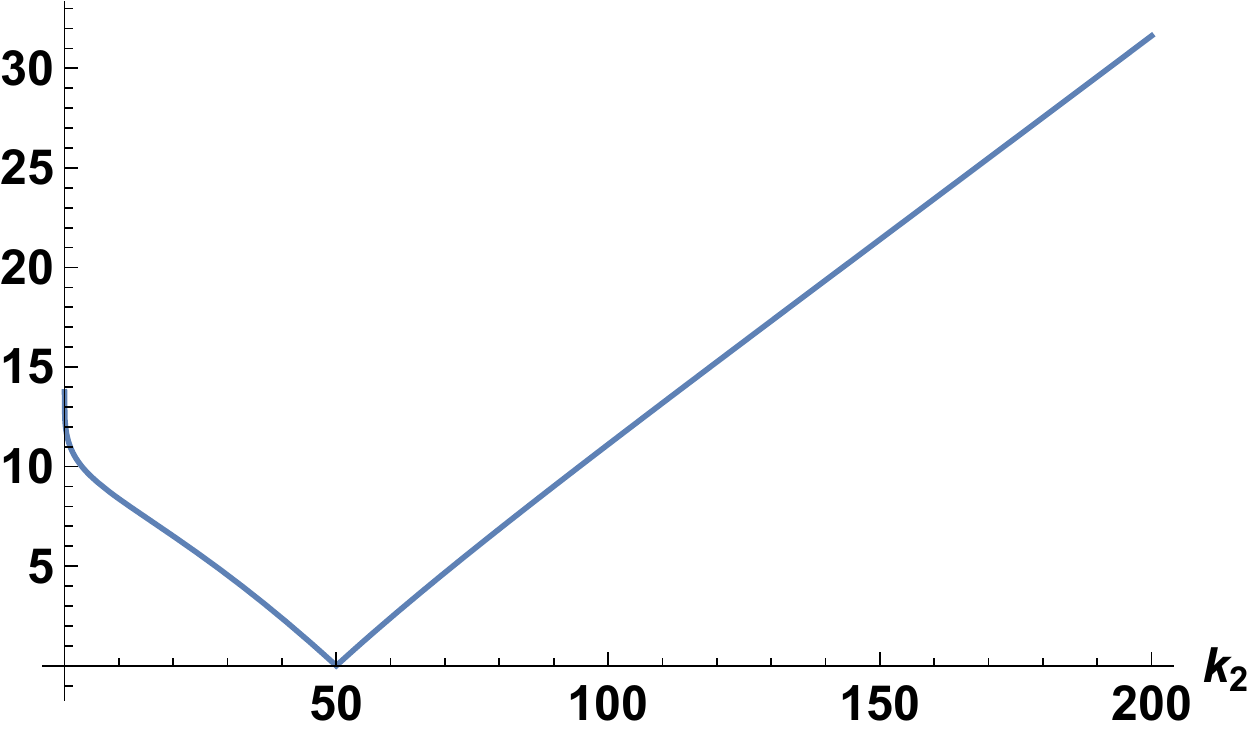}
\subcaption{When $Z \sim \mathcal{N}(0, \frac{1}{2})$}
\label{fig:fall}
\end{center}
\end{minipage}
\end{tabular}
\caption{Plots of $\Ex[\Pe(Z)] - \Ex[\Pe(Z + C)]$ for the Laplace and the Gauss distributions}
\end{figure}

\section{An inequality for the variance of the loss}
In this section, we derive an inequality for the variance of $\Pe(Z + c)$. 
Let $C$ be the value of $c$ giving the minimum value of $\Ex[\Pe(Z + c)]$. 
Then, the following holds: 

\begin{thm}[Section~$\ref{S1}$, Theorem~$\ref{thm:1.3}$]\label{thm:3.4}
We have 
\begin{align*}
\Va[\Pe(Z + C)] \leq \Va[\Pe(Z)], 
\end{align*}
where equality holds only when $C = 0$; that is, when $k_{1} = k_{2}$.  
\end{thm}

Fixing the conditions as $k_{1} = 50$ and $b = 1$, 
we can plot $\Va[\Pe(Z)] - \Va[\Pe(Z + C)]$ for the Laplace and the Gauss distributions as follows: 
\begin{figure}[ht!]
\begin{tabular}{cc}
\begin{minipage}{0.5\hsize}
\begin{center}
\includegraphics[height=4cm]{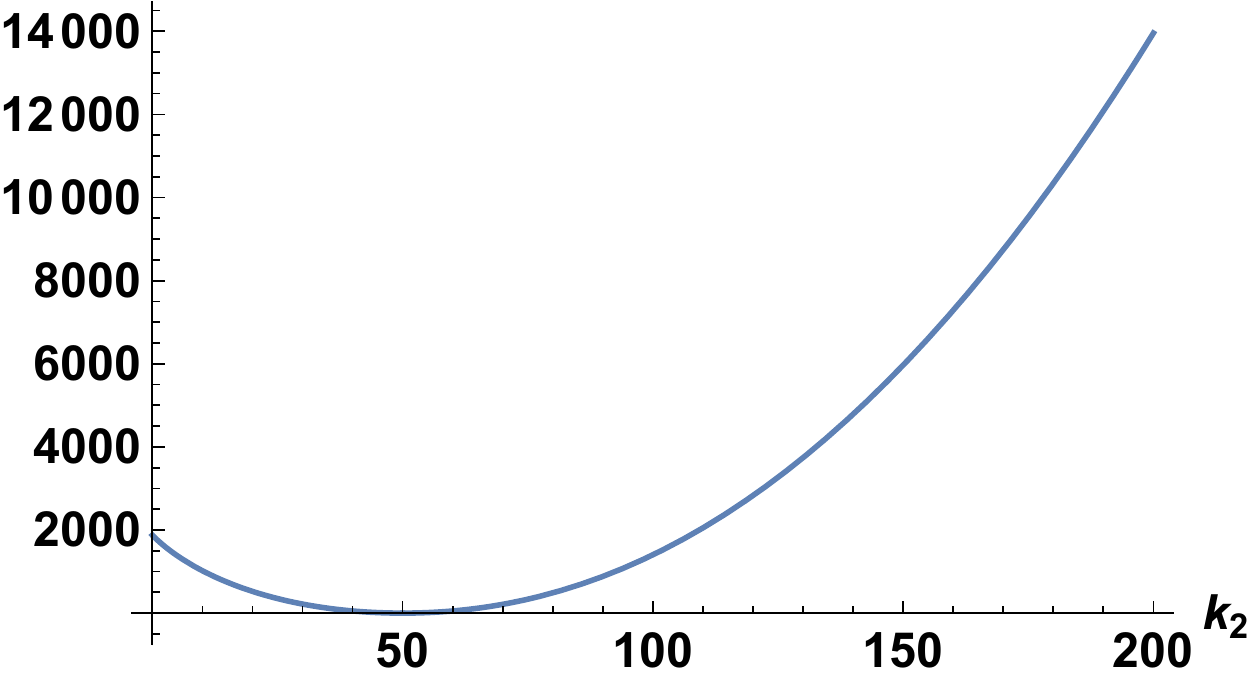}
\subcaption{When $Z \sim {\rm Laplace}(0, 1)$}
\label{fig:winter}
\end{center}
\end{minipage}
\begin{minipage}{0.5\hsize}
\begin{center}
\includegraphics[height=4cm]{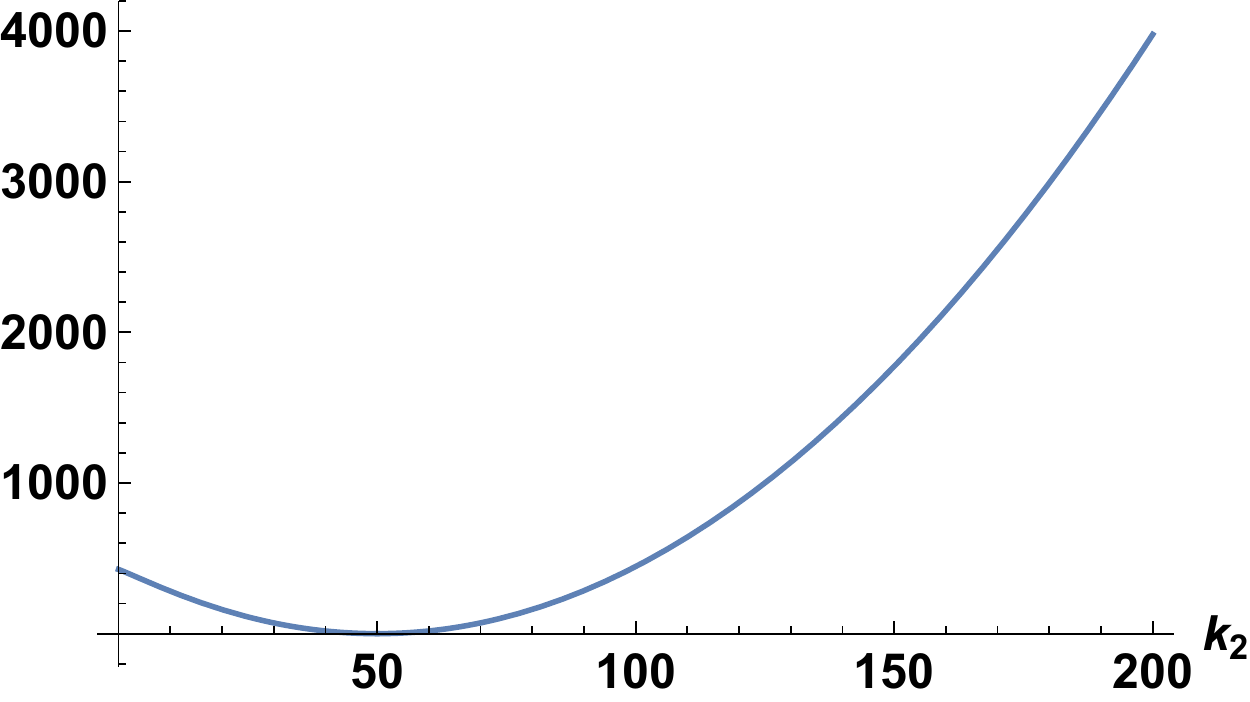}
\subcaption{When $Z \sim \mathcal{N}(0, \frac{1}{2})$}
\label{fig:fall}
\end{center}
\end{minipage}
\end{tabular}
\caption{Plots of $\Va[\Pe(Z)] - \Va[\Pe(Z + C)]$ for the Laplace and the Gauss distributions}
\end{figure}

To prove Theorem~$\ref{thm:3.4}$, we use the following lemma: 
\begin{lem}[Section~$\ref{S1}$, Lemma~$\ref{lem:1.4}$]\label{lem:3.5}
For $a > 0$ and $x > 0$, 
we have 
\begin{align*}
x^{a} \g(a, x)^{2} - x^{a} \G(a)^{2} + 2 \g(a, x) \G(2a, x) > 0. 
\end{align*}
\end{lem}

The proof of Lemma~$\ref{lem:3.5}$ is presented in Section~$5.2$. 
Now we can prove Theorem~$\ref{thm:3.4}$. 

\begin{proof}[Proof of Theorem~$\ref{thm:3.4}$]
It follows from the equation $(\ref{C})$ that
\begin{align*}
\frac{(k_{1}^{2} - k_{2}^{2}) b C}{2 \G(a)} 
\G\left(2a, \left\lvert \frac{C}{b} \right\rvert^{\frac{1}{a}} \right)
&= - \Sgn(C) \frac{k_{2} - k_{1}}{k_{1} + k_{2}} 
\frac{(k_{1} + k_{2})^{2} b \lvert C \rvert}{2 \G(a)}
\G\left(2a, \left\lvert \frac{C}{b} \right\rvert^{\frac{1}{a}} \right) \\
&= - \frac{(k_{1} + k_{2})^{2} b \lvert C \rvert}{2 \G(a)^{2}} 
\g\left(a, \left\lvert \frac{C}{b} \right\rvert^{\frac{1}{a}} \right) 
\G\left(2a, \left\lvert \frac{C}{b} \right\rvert^{\frac{1}{a}} \right), \allowdisplaybreaks \\
\Sgn(C) \frac{(k_{1}^{2} - k_{2}^{2}) b^{2}}{2 \G(a)} 
\g\left(3a, \left\lvert \frac{C}{b} \right\rvert^{\frac{1}{a}} \right) 
&= - \Sgn(C) \frac{k_{2} - k_{1}}{k_{1} + k_{2}} 
\frac{(k_{1} + k_{2})^{2} b^{2}}{2 \G(a)} 
\g\left(3a, \left\lvert \frac{C}{b} \right\rvert^{\frac{1}{a}} \right) \\
&= - \frac{(k_{1} + k_{2})^{2} b^{2}}{2 \G(a)^{2}}  
\g\left(a, \left\lvert \frac{C}{b} \right\rvert^{\frac{1}{a}} \right) 
\g\left(3a, \left\lvert \frac{C}{b} \right\rvert^{\frac{1}{a}} \right). 
\end{align*}
Hence, substituting $c = C$ in equation~$(2)$ of Lemma~$\ref{lem:3.1}$, 
we have 
\begin{align*}
\Va[\Pe(Z + C)] 
&= \frac{(k_{1} + k_{2})^{2} C^{2}}{4} 
- \frac{(k_{1} + k_{2})^{2} b \lvert C \rvert}{\G(a)^{2}} 
\g\left(a, \left\lvert \frac{C}{b} \right\rvert^{\frac{1}{a}} \right) 
\G\left(2a, \left\lvert \frac{C}{b} \right\rvert^{\frac{1}{a}} \right) \\
&\quad - \frac{(k_{1} + k_{2})^{2} C^{2}}{4 \G(a)^{2}}
\g\left(a, \left\lvert \frac{C}{b} \right\rvert^{\frac{1}{a}} \right)^{2} 
- \frac{(k_{1} + k_{2})^{2} b^{2}}{4 \G(a)^{2}}
\G\left(2a, \left\lvert \frac{C}{b} \right\rvert^{\frac{1}{a}} \right)^{2} \\
&\quad + \frac{(k_{1}^{2} + k_{2}^{2}) b^{2} \G(3a)}{2 \G(a)}  
- \frac{(k_{1} + k_{2})^{2} b^{2}}{2 \G(a)^{2}}  
\g\left(a, \left\lvert \frac{C}{b} \right\rvert^{\frac{1}{a}} \right) 
\g\left(3a, \left\lvert \frac{C}{b} \right\rvert^{\frac{1}{a}} \right). 
\end{align*}
From this and equation $(\ref{V[L(Z)]})$, we obtain 
\begin{align*}
&\!\!\!\!\Va[\Pe(Z)] - \Va[\Pe(Z + C)] \\
&= - \frac{(k_{1} + k_{2})^{2} C^{2}}{4}  
+ \frac{(k_{1} + k_{2})^{2} b \lvert C \rvert}{\G(a)^{2}} 
\g\left(a, \left\lvert \frac{C}{b} \right\rvert^{\frac{1}{a}} \right) 
\G\left(2a, \left\lvert \frac{C}{b} \right\rvert^{\frac{1}{a}} \right) \\
&\quad + \frac{(k_{1} + k_{2})^{2} C^{2}}{4 \G(a)^{2}}
\g\left(a, \left\lvert \frac{C}{b} \right\rvert^{\frac{1}{a}} \right)^{2} 
+ \frac{(k_{1} + k_{2})^{2} b^{2}}{4 \G(a)^{2}}
\G\left(2a, \left\lvert \frac{C}{b} \right\rvert^{\frac{1}{a}} \right)^{2} \\
&\quad - \frac{(k_{1} + k_{2})^{2} b^{2} \G(2a)^{2}}{4 \G(a)^{2}} 
+ \frac{(k_{1} + k_{2})^{2} b^{2}}{2 \G(a)^{2}}  
\g\left(a, \left\lvert \frac{C}{b} \right\rvert^{\frac{1}{a}} \right) 
\g\left(3a, \left\lvert \frac{C}{b} \right\rvert^{\frac{1}{a}} \right) \\
&= \frac{(k_{1} + k_{2})^{2} b^{2}}{4 \G(a)^{2}} 
f \left(a, \left\lvert \frac{C}{b} \right\rvert^{\frac{1}{a}} \right),  
\end{align*}
where, for $a > 0$ and $x \geq 0$, $f(a, x)$ is defined as  
\begin{align*}
f(a, x) 
&:= x^{2a} \g(a, x)^{2} - x^{2a} \G(a)^{2} + 4 x^{a} \g(a, x) \G(2a, x) \\
&\quad + \G(2a, x)^{2} - \G(2a)^{2} + 2 \g(a, x) \g(3a, x). 
\end{align*}
Here, since  
\begin{align*}
\frac{d}{dx} f(a, x) 
&= 2 a x^{a - 1} \left\{x^{a} \g(a, x)^{2} - x^{a} \G(a)^{2} + 2\g(a, x) \G(2a, x) \right\} \\
&\quad + 2 x^{a - 1} e^{-x} \g(3a, x) + 2 x^{2a - 1} e^{-x} \G(2a, x), 
\end{align*}
from Lemma~$\ref{lem:3.5}$, we have $\frac{d}{dx} f(a, x) > 0$ ($a > 0$, $x > 0$). 
Also, $f(a, 0) = 0$ holds for $a > 0$. 
Therefore, we obtain 
\begin{align*}
\Va[\Pe(Z)] - \Va[\Pe(Z + C)] \geq 0,  
\end{align*}
where equality holds only when $C = 0$. 
Moreover, from equation~$(\ref{C})$, we find that $C = 0$ holds only when $k_{1} = k_{2}$. 
\end{proof}

\section{Inequalities for the gamma and the incomplete gamma functions}
In this section, we give some inequalities for the gamma and the incomplete gamma functions, 
which we used to derive the inequality for the variance of the loss in Theorem~$\ref{thm:3.4}$. 

\subsection{Inequalities for the gamma function}

To prove Lemma~$\ref{lem:3.5}$, we use the following: 
\begin{lem}[Section~$\ref{S1}$, Lemma~$\ref{lem:1.5}$]\label{lem:4-1-3}
For $a > 0$, we have 
\begin{align*}
2 \G(2a) - a \G(a)^{2} > 0. 
\end{align*}
\end{lem}

Next, to prove Lemma~$\ref{lem:4-1-3}$, we use the following:
\begin{lem}[Section~$\ref{S1}$, Lemma~$\ref{lem:1.6}$]\label{lem:4-1-1}
For $a > 0$, we have 
\begin{align*}
4^{a} \G\left(a + \frac{1}{2} \right) > \sqrt{\pi} \G(a + 1). 
\end{align*}
\end{lem}

Furthermore, to prove Lemma~$\ref{lem:4-1-1}$, we need another lemma: 
\begin{lem}\label{lem:4-1-2}
We have 
\begin{align*}
\sum_{n = 1}^{\infty} \frac{1}{n (2n - 1)} = 2 \log{2}. 
\end{align*}
\end{lem}
\begin{proof}
Let $S_{n} := \sum_{k = 1}^{n} \frac{1}{k (2k - 1)}$. 
Accordingly, we have 
\begin{align*}
S_{n} 
&= \sum_{k = 1}^{n} \left(\frac{2}{2k - 1} - \frac{1}{k} \right) \allowdisplaybreaks \\
&= 2 \sum_{k = 1}^{n} \frac{1}{2k - 1} - \sum_{k = 1}^{n} \frac{1}{k} \allowdisplaybreaks \\
&= 2 \sum_{k = 1}^{n} \frac{1}{2k - 1} 
+ \left(2 \sum_{k = 1}^{n} \frac{1}{2 k} - 2 \sum_{k = 1}^{n} \frac{1}{2 k} \right) 
- \sum_{k = 1}^{n} \frac{1}{k} \allowdisplaybreaks \\
&= 2 \left(\sum_{k = 1}^{n} \frac{1}{2 k - 1}  + \sum_{k = 1}^{n} \frac{1}{2 k} \right) 
- 2 \sum_{k = 1}^{n} \frac{1}{k} \allowdisplaybreaks \\
&= 2 \sum_{k = 1}^{2n} \frac{1}{k} - 2 \sum_{k = 1}^{n} \frac{1}{k} \allowdisplaybreaks \\
&= 2 \sum_{k = n + 1}^{2n} \frac{1}{k} \allowdisplaybreaks \\
&= 2 \sum_{k = 1}^{n} \frac{1}{k + n}. 
\end{align*}
Therefore, we find 
\begin{align*}
\lim_{n \rightarrow \infty} S_{n} 
&= 2 \lim_{n \to \infty} \sum_{k = 1}^{n} \frac{1}{k + n} \allowdisplaybreaks \\
&= 2 \lim_{n \to \infty} \frac{1}{n} \sum_{k = 1}^{n} \frac{1}{1 + \frac{k}{n}} 
\allowdisplaybreaks \\
&= \int_{0}^{1} \frac{1}{1 + x} dx \allowdisplaybreaks \\
&= 2\log{2}.
\end{align*}
The lemma is thus proved. 
\end{proof}

Now we can prove Lemma~$\ref{lem:4-1-1}$. 
\begin{proof}[Proof of Lemma~$\ref{lem:4-1-1}$]
Let 
\begin{align*}
g(a) := \frac{4^{a} \G\left(a + \frac{1}{2} \right)}{\sqrt{\pi} \G(a + 1)}. 
\end{align*}
To prove $g(a) > 1$ for $a > 0$, 
we use the following formula \cite[p.13, Theorem~1.2.5]{andrews_askey_roy_1999}: 
\begin{align*}
\frac{d}{dx} \log{\G(x)} 
= \frac{\G^{'}(x)}{\G(x)} 
= - \g_{0} + \sum_{n = 1}^{\infty} \left(\frac{1}{n} - \frac{1}{x + n - 1} \right), 
\end{align*}
where $\g_{0}$ is Euler's constant given by 
\begin{align*}
\g_{0} := \lim_{n \to \infty} \left(\sum_{k = 1}^{n} \frac{1}{k} - \log{n} \right). 
\end{align*}
Taking the logarithmic derivative of $g(a)$, from the above formula, we have 
\begin{align*}
\frac{d}{da} \log{g(a)} 
&= 2\log{2} + \frac{d}{da} \log{\G\left(a + \frac{1}{2} \right)} 
- \frac{d}{da} \log{\G(a + 1)} \allowdisplaybreaks \\
&= 2\log{2} + \sum_{n = 1}^{\infty} \left(\frac{1}{n} - \frac{1}{a - \frac{1}{2} + n} \right) 
- \sum_{n = 1}^{\infty} \left(\frac{1}{n} - \frac{1}{a + n} \right) \allowdisplaybreaks \\
&= 2\log{2} - \frac{1}{2} \sum_{n = 1}^{\infty} \frac{1}{(a + n) \left(a - \frac{1}{2} + n \right)} 
\allowdisplaybreaks \\
&> 2\log{2} - \frac{1}{2} \sum_{n = 1}^{\infty} \frac{1}{n \left(n - \frac{1}{2} \right)} 
\allowdisplaybreaks \\
&= 2\log{2} - \sum_{n = 1}^{\infty} \frac{1}{n (2n - 1)}
\end{align*}
for $a > 0$. 
Moreover, using Lemma~$\ref{lem:4-1-2}$, we obtain $\frac{d}{da} \log{g(a)} > 0$ for $a > 0$. 
This leads to $\frac{d}{da} g(a) > 0$ for $a > 0$. 
The lemma follows from this and $g(0) = 1$.  
\end{proof}

Now, we can prove Lemma~$\ref{lem:4-1-3}$. 
\begin{proof}[Proof of Lemma~$\ref{lem:4-1-3}$]
We use the following formula \cite[p.22, Theorem~6.5.1]{andrews_askey_roy_1999}:
\begin{align*}
\G(2a) = \frac{2^{2a - 1}}{\sqrt{\pi}} \G(a) \G\left(a + \frac{1}{2} \right). 
\end{align*}
From this and Lemma~$\ref{lem:4-1-1}$, we have 
\begin{align*}
2 \G(2a) - a \G(a)^{2} 
&= \frac{2^{2a}}{\sqrt{\pi}} \G(a) \G\left(a + \frac{1}{2} \right) 
- \G(a) \G(a + 1) \allowdisplaybreaks \\
&= \frac{1}{\sqrt{\pi}} \G(a) 
\left\{4^{a} \G\left(a + \frac{1}{2} \right) - \sqrt{\pi} \G(a + 1) \right\} \\
&> 0. 
\end{align*}
The lemma is thus proved. 
\end{proof}

\subsection{Inequalities for the incomplete gamma functions}
We will prove the following lemma: 
\begin{lem}[Section~$4$, Lemma~$\ref{lem:3.5}$]\label{lem:4-2-1}
For $a > 0$ and $x > 0$, we have 
\begin{align*}
x^{a} \g(a, x)^{2} - x^{a} \G(a)^{2} + 2 \g(a, x) \G(2a, x) > 0. 
\end{align*}
\end{lem}

To prove Lemma~$\ref{lem:4-2-1}$, we need to prove two other lemmas: 
\begin{lem}\label{lem:4-2-2}
For $a > 0$ and $x \geq 0$, we have 
\begin{align*}
a \g(a, x) \geq x^{a} e^{-x}. 
\end{align*}
\end{lem}
\begin{proof}
For $a > 0$ and $x \geq 0$, we define 
\begin{align*}
u(a, x) := a \g(a, x) - x^{a} e^{-x}. 
\end{align*}
Then, we have 
\begin{align*}
\frac{d}{dx} u(a, x) = x^{a} e^{-x} \geq 0. 
\end{align*}
The lemma follows from this and $u(a, 0) = 0$. 
\end{proof}

\begin{lem}\label{lem:4-2-3}
For $a > 0$ and $b \in \mathbb{R}$, we have 
\begin{align*}
\lim_{x \to +\infty} x^{b} \G(a, x) = 0. 
\end{align*}
\end{lem}
\begin{proof}
When $b \leq 0$, it is easily obtained from the definition of $\G(a, x)$. 
When $b > 0$, using the L'H\^opital's rule, we obtain 
\begin{align*}
\lim_{x \to +\infty} \frac{\G(a, x)}{x^{-b}} 
&= \lim_{x \to +\infty} \frac{x^{a - 1} e^{-x}}{b x^{- b - 1}} \allowdisplaybreaks \\
&= \lim_{x \to +\infty} \frac{x^{a + b} e^{-x}}{b} \\
&= 0. 
\end{align*}
\end{proof}

Now, we can prove Lemma~$\ref{lem:4-2-1}$. 
\begin{proof}[Proof of Lemma~$\ref{lem:4-2-1}$]
For $a > 0$ and $x \geq 0$, we define 
\begin{align*}
y_{1} (a, x) := x^{a} \g(a, x)^{2} - x^{a} \G(a)^{2} + 2 \g(a, x) \G(2a, x). 
\end{align*}
Let us prove $y_{1} (a, x) > 0$ ($a > 0$, $x > 0$). 
For $a > 0$ and $x \geq 0$, we define 
\begin{align*}
y_{2} (a, x) &:= a \g(a, x)^{2} - a \G(a)^{2} + 2 e^{-x} \G(2a, x); \\
y_{3} (a, x) &:= a x^{a - 1} \g(a, x) - \G(2a, x) - x^{2a - 1} e^{-x}; \\
y_{4} (a, x) &:= a (a - 1) \g(a, x) + x^{a} e^{-x} (2x + 1 - a). 
\end{align*}
Then, we have 
\begin{align*}
\frac{d}{dx} y_{1} (a, x) &= x^{a - 1} y_{2} (a, x); \\
\frac{d}{dx} y_{2} (a, x) &= 2 e^{-x} y_{3} (a, x); \\
\frac{d}{dx} y_{3} (a, x) &= x^{a - 2} y_{4} (a, x); \\
\frac{d}{dx} y_{4} (a, x) &= x^{a} e^{-x} (3a + 1 - 2x). 
\end{align*}
From these relations, we find that 
the (positive or negative) signs of $\frac{d}{dx} y_{i}(a, x)$ and $y_{i + 1}(a, x)$ ($i = 1, 2, 3$) 
are equal to each other for $a > 0$ and $x > 0$. 
Let $p_{i} (a)$ ($i = 2, 3, 4$) be the value of $x$ satisfying $y_{i}(a, x) = 0$. 
It is easily verified that $\lim_{x \to 0+} \frac{d}{dx}y_{4}(a, x) = \lim_{x \to +\infty}\frac{d}{dx}y_{4}(a, x) = \lim_{x \to 0+} y_{4}(a, x) = 0$ and $\lim_{x \to +\infty}y_{4}(a, x) = a (a - 1) \G(a)$ for $a > 0$.  
Therefore, from the first derivative test, we obtain Tables~$1$ and $2$. 
Moreover, using Lemmas~$\ref{lem:4-2-2}$, $\ref{lem:4-2-3}$, and L'H\^opital's rule, 
we obtain 
\begin{align*}
&\lim_{x \to 0+} \frac{d}{dx} y_{3}(a, x) 
= 
\begin{cases}
\infty & (0 < a < 1), \\
0 & (a \geq 1),
\end{cases}
& 
&\lim_{x \to +\infty} \frac{d}{dx} y_{3}(a, x) 
= 
\begin{cases}
0 & (0 < a < 2), \\
2 & (a = 2), \\
\infty & (a > 2), 
\end{cases} \allowdisplaybreaks \\
&\lim_{x \to 0+} y_{3}(a, x) 
= - \G(2a) \quad (a > 0), 
& 
&\lim_{x \to +\infty} y_{3}(a, x) 
= 
\begin{cases}
0 & (0 < a < 1), \\
1 & (a = 1), \\
\infty & (a > 1), 
\end{cases} \allowdisplaybreaks \\
&\lim_{x \to 0+} \frac{d}{dx} y_{2}(a, x) 
= - 2\G(2a) \quad (a > 0), 
& 
&\lim_{x \to +\infty} \frac{d}{dx} y_{2}(a, x) 
= 0 \quad (a > 0), \allowdisplaybreaks \\
&\lim_{x \to 0+} y_{2}(a, x) 
= 2 \G(2a) - a\G(a)^{2} \quad (a > 0), 
&
&\lim_{x \to +\infty} y_{2}(a, x) 
= 0 \quad (a > 0), \allowdisplaybreaks \\
&\lim_{x \to 0+} \frac{d}{dx} y_{1}(a, x) 
= 
\begin{cases}
\infty & (0 < a < 1), \\
1 & (a = 1), \\
0 & (a > 1), 
\end{cases}
& 
&\lim_{x \to +\infty} \frac{d}{dx} y_{1}(a, x) 
= 0 \quad (a > 0), \allowdisplaybreaks \\
&\lim_{x \to 0+} y_{1}(a, x) 
= 0 \quad (a > 0), 
& 
&\lim_{x \to +\infty} y_{1}(a, x) 
= 0 \quad (a > 0). 
\end{align*}
From these results, Lemma~$\ref{lem:4-1-3}$, and the fact that 
the signs of $\frac{d}{dx} y_{i}(a, x)$ and $y_{i + 1}(a, x)$ ($i = 1, 2, 3$) 
are equal to each other for $a > 0$ and $x > 0$, 
we obtain Tables~$3$ and $4$. 
From Tables~$3$ and $4$, we can verify that $y_{1}(a, x) > 0$ holds for $a > 0$ and $x > 0$. 
This completes the proof of the lemma. 
\end{proof}

\begin{table}[htb] 
\begin{center}
\caption{Case of $0 < a < 1$}
\begin{tabular}{|c|c|c|c|c|c|c|c|}
\hline
$x$ & $\;0\;$ &$\cdots$ & $\;\frac{3a + 1}{2}\;$ &$\cdots$ & $\;p_{4}(a)\;$ & 
$\cdots$ & $+\infty$ 
\\ \hline 
$\frac{d}{dx} y_{4}(a, x)$ & $0$ & $+$ & $0$ & $-$ & $-$ & $-$ & $0$ 
\\ \hline 
$y_{4}(a, x)$ & $0$ & $+$ & $+$ & $+$ & $0$ & $-$ & $-$ 
\\ \hline 
\end{tabular}
\end{center}
\end{table}
\begin{table}[htb]
\begin{center}
\caption{Case of $a \geq 1$}
\begin{tabular}{|c|c|c|c|c|c|}
\hline
$x$ & $\;0\;$ &$\cdots$ & $\;\frac{3a + 1}{2}\;$ &$\cdots$ & $\;+\infty\;$ 
\\ \hline 
$\frac{d}{dx} y_{4}(a, x)$ & $0$ & $+$ & $0$ & $-$ & $0$ 
\\ \hline 
$y_{4}(a, x)$ & $0$ & $+$ & $+$ & $+$ & 
$\begin{matrix}0\;\;(a = 1)\\ +\;(a > 1)\end{matrix}$ 
\\ \hline 
\end{tabular}
\end{center}
\end{table}
\begin{table}[htb]
\begin{center}
\caption{Case of $0 < a < 1$}
\begin{tabular}{|c|c|c|c|c|c|c|c|c|c|}
\hline
$x$ & $\;0\;$ & $\cdots$ & $\;p_{2}(a)\;$ & $\cdots$ & $\;p_{3}(a)\;$ 
& $\cdots$ & $\;p_{4}(a)\;$ & $\cdots$ & $\;+\infty\;$ 
\\ \hline 
$\frac{d}{dx} y_{3}(a, x)$ & $+\infty$ & $+$ & $+$ & $+$ & $+$ & $+$ & $0$ & $-$ & $0$ 
\\ \hline 
$y_{3}(a, x)$ & $-$ & $-$ & $-$ & $-$ & $0$ & $+$ & $+$ & $+$ & $0$  
\\ \hline 
$\frac{d}{dx} y_{2}(a, x)$ & $-$ & $-$ & $-$ & $-$ & $0$ 
& $+$ & $+$ & $+$ & $0$ 
\\ \hline 
$y_{2}(a, x)$ & $+$ & $+$ & $0$ & $-$ & $-$ & $-$ & $-$ & $-$ & $0$ 
\\ \hline 
$\frac{d}{dx} y_{1}(a, x)$ & $+$ & $+$ & $0$ 
& $-$ & $-$ & $-$ & $-$ & $-$ & $0$ 
\\ \hline 
$y_{1}(a, x)$ & $0$ & $+$ & $+$ & $+$ & $+$ & $+$ & $+$ & $+$ & $0$ 
\\ \hline 
\end{tabular}
\end{center}
\end{table}
\begin{table}[htb]
\begin{center}
\caption{Case of $a \geq 1$}
\begin{tabular}{|c|c|c|c|c|c|c|c|}
\hline
$x$ & $\;0\;$ & $\cdots$ & $\;p_{2}(a)\;$ & $\cdots$ & $\;p_{3}(a)\;$ 
& $\cdots$ & $\;+\infty\;$ 
\\ \hline 
$\frac{d}{dx} y_{3}(a, x)$ & $0$ & $+$ & $+$ & $+$ & $+$ & $+$ & 
$\begin{matrix}0\;(a < 2)\\ +\;(a = 2)\\ +\infty\;(a > 2)\end{matrix}$ 
\\ \hline 
$y_{3}(a, x)$ & $-$ & $-$ & $-$ & $-$ & $0$ & $+$ & $0$  
\\ \hline 
$\frac{d}{dx} y_{2}(a, x)$ & $-$ & $-$ & $-$ & $-$ & $0$ 
& $+$ & $0$ 
\\ \hline 
$y_{2}(a, x)$ & $+$ & $+$ & $0$ & $-$ & $-$ & $-$ & $0$ 
\\ \hline 
$\frac{d}{dx} y_{1}(a, x)$ & $+$ & $+$ & $0$ 
& $-$ & $-$ & $-$ & $0$ 
\\ \hline 
$y_{1}(a, x)$ & $0$ & $+$ & $+$ & $+$ & $+$ & $+$ & $0$ 
\\ \hline 
\end{tabular}
\end{center}
\end{table}

\section{Calculation of the expected value and the variance of the loss}
Here, 
we calculate the expected value and the variance of the loss $\Pe(Z + c)$ for $c \in \mathbb{R}$.

\subsection{\bf{Expected value of the loss}}
%We calculate the expected value of the loss $\Pe(Z + c)$ for $c \in \mathbb{R}$. 
Here, let us put $\beta := (2 a b \G(a))^{-1}$; then, we have 
\begin{align*}
\Ex[\Pe(Z + c)] 
&= \int_{- \infty}^{+\infty} \Pe(z + c) f_{Z}(z) dz \\ 
&= k_{2} \beta \int_{- \infty}^{- c} (- z - c) 
\exp{\left( - \left\lvert \frac{z}{b} \right\rvert^{\frac{1}{a}} \right)} dz
+ k_{1} \beta \int_{- c}^{+\infty} (z + c) 
\exp{\left( - \left\lvert \frac{z}{b} \right\rvert^{\frac{1}{a}} \right)} dz. 
\end{align*}
Replace $z$ with $b z$ to get 
\begin{align*} 
\Ex[\Pe(Z + c)]
= k_{2} b \beta \int_{- \infty}^{- c / b} (- b z - c) 
\exp{\left( - \lvert z \rvert^{\frac{1}{a}} \right)} dz
+ k_{1} b \beta \int_{- c / b}^{+\infty} (b z + c) 
\exp{\left( - \lvert z \rvert^{\frac{1}{a}} \right)} dz. 
\end{align*}
When $c \geq 0$,  we have 
\begin{align*}
\Ex[\Pe(Z + c)] 
&= k_{2} b \beta 
\int_{- \infty}^{- c / b} (- b z - c) \exp{\left( - (-z)^{\frac{1}{a}} \right)} dz \\
&\quad + k_{1} b \beta 
\int_{- c / b}^{0} (b z + c) \exp{\left( - (-z)^{\frac{1}{a}} \right)} dz
+ k_{1} b \beta \int_{0}^{+\infty} (b z + c) \exp{\left( - z^{\frac{1}{a}} \right)} dz 
\allowdisplaybreaks \\
&= k_{2} b \beta 
\int_{c / b}^{+\infty} (b z - c) \exp{\left( - z^{\frac{1}{a}} \right)} dz \\
&\quad + k_{1} b \beta 
\int_{0}^{c / b} (- b z + c) \exp{\left( - z^{\frac{1}{a}} \right)} dz
+ k_{1} b \beta \int_{0}^{+\infty} (b z + c) \exp{\left( - z^{\frac{1}{a}} \right)} dz 
\allowdisplaybreaks \\
&= (k_{1} + k_{2}) b^{2} \beta 
\int_{c / b}^{+\infty} z \exp{\left( - z^{\frac{1}{a}} \right)} dz \\
&\quad + (k_{1} - k_{2}) b c \beta 
\int_{0}^{+\infty} \exp{\left( - z^{\frac{1}{a}} \right)} dz
+ (k_{1} + k_{2}) b c \beta \int_{0}^{c / b} \exp{\left( - z^{\frac{1}{a}} \right)} dz.  
\end{align*}
When $c < 0$,  we have 
\begin{align*}
\Ex[\Pe(Z + c)] 
&= k_{2} b \beta 
\int_{- \infty}^{0} (- b z - c) \exp{\left( - (-z)^{\frac{1}{a}} \right)} dz 
+ k_{2} b \beta \int_{0}^{- c / b} (- b z - c) \exp{\left( - z^{\frac{1}{a}} \right)} dz \\
&\quad + k_{1} b \beta 
\int_{- c / b}^{+\infty} (b z + c) \exp{\left( - z^{\frac{1}{a}} \right)} dz 
\allowdisplaybreaks \\
&= k_{2} b \beta \int_{0}^{+\infty} (b z - c) \exp{\left( - z^{\frac{1}{a}} \right)} dz 
+ k_{2} b \beta \int_{0}^{- c / b} (- b z - c) \exp{\left( - z^{\frac{1}{a}} \right)} dz \\
&\quad + k_{1} b \beta 
\int_{- c / b}^{+\infty} (b z + c) \exp{\left( - z^{\frac{1}{a}} \right)} dz 
\allowdisplaybreaks \\
&= (k_{1} + k_{2}) b^{2} \beta 
\int_{- c / b}^{+\infty} z \exp{\left( - z^{\frac{1}{a}} \right)} dz \\
&\quad + (k_{1} - k_{2}) b c \beta 
\int_{0}^{+\infty} \exp{\left( - z^{\frac{1}{a}} \right)} dz
- (k_{1} + k_{2}) b c \beta \int_{0}^{- c / b} \exp{\left( - z^{\frac{1}{a}} \right)} dz.  
\end{align*}
From the above, for any $c \in \mathbb{R}$, we have 
\begin{align*}
\Ex[\Pe(Z + c)] 
&= (k_{1} + k_{2}) b^{2} \beta 
\int_{\lvert c  / b \rvert}^{+\infty} z \exp{\left( - z^{\frac{1}{a}} \right)} dz \\
&\quad + (k_{1} - k_{2}) b c \beta 
\int_{0}^{+\infty} \exp{\left( - z^{\frac{1}{a}} \right)} dz
+ (k_{1} + k_{2}) b \lvert c \rvert \beta 
\int_{0}^{\lvert c / b \rvert} \exp{\left( - z^{\frac{1}{a}} \right)} dz.  
\end{align*}
Now set $t := z^{\frac{1}{a}}$ to get 
\begin{align*}
\Ex[\Pe(Z + c)] 
&= (k_{1} + k_{2}) a b^{2} \beta \int_{c'}^{+\infty} t^{2a - 1} e^{-t} dt \\
&\quad + (k_{1} - k_{2}) a b c \beta \int_{0}^{+\infty} t^{a - 1} e^{-t} dt
+ (k_{1} + k_{2}) a b \lvert c \rvert \beta \int_{0}^{c'} t^{a - 1} e^{-t} dt 
\allowdisplaybreaks \\
&= (k_{1} + k_{2}) a b^{2} \beta \G(2a, c') 
+ (k_{1} - k_{2}) a b c \beta \G(a) 
+ (k_{1} + k_{2}) a b \lvert c \rvert \beta \g(a, c'), 
\end{align*}
where $c' := \lvert c / b \rvert^{\frac{1}{a}}$. 
Therefore, for any $c \in \mathbb{R}$, we have 
\begin{align}\label{E[Pe(Z+c)]-2}
\Ex[\Pe(Z + c)] 
= \frac{(k_{1} - k_{2}) c}{2} 
+ \frac{(k_{1} + k_{2}) \lvert c \rvert}{2 \G(a)} 
\g\left(a, \left\lvert \frac{c}{b} \right\rvert^{\frac{1}{a}} \right) 
+ \frac{(k_{1} + k_{2}) b}{2 \G(a)} 
\G\left(2a, \left\lvert \frac{c}{b} \right\rvert^{\frac{1}{a}} \right).    
\end{align}

\subsection{\bf{Variance of the loss}}
Now let us calculate the variance of the loss $\Pe(Z + c)$ for $c \in \mathbb{R}$. 
Put $\beta := (2 a b \G(a))^{-1}$; then, we have 
\begin{align*}
\Ex[\Pe(Z + c)^{2}] 
&= \int_{- \infty}^{+\infty} \Pe(z + c)^{2} f_{Z}(z) dz 
\allowdisplaybreaks \\
&= k_{2}^{2} \beta \int_{- \infty}^{- c} (z + c)^{2} 
\exp{\left( - \left\lvert \frac{z}{b} \right\rvert^{\frac{1}{a}} \right)} dz
+ k_{1}^{2} \beta \int_{- c}^{+\infty} (z + c)^{2} 
\exp{\left( - \left\lvert \frac{z}{b} \right\rvert^{\frac{1}{a}} \right)} dz. 
\end{align*}
Replace $z$ with $b z$ to get
\begin{align*}
\Ex[\Pe(Z + c)^{2}] 
= k_{2}^{2} b \beta \int_{- \infty}^{- c / b} (b z + c)^{2} 
\exp{\left( - \lvert z \rvert^{\frac{1}{a}} \right)} dz
+ k_{1}^{2} b \beta \int_{- c / b}^{+\infty} (b z + c)^{2} 
\exp{\left( - \lvert z \rvert^{\frac{1}{a}} \right)} dz. 
\end{align*}
When $c \geq 0$,  we have 
\begin{align*}
\Ex[\Pe(Z + c)^{2}] 
&= k_{2}^{2} b \beta \int_{- \infty}^{- c / b} (b z + c)^{2} 
\exp{\left( - (- z)^{\frac{1}{a}} \right)} dz \\
&\quad + k_{1}^{2} b \beta \int_{- c / b}^{0} (b z + c)^{2} 
\exp{\left( - (- z)^{\frac{1}{a}} \right)} dz
+ k_{1}^{2} b \beta \int_{0}^{+\infty} (b z + c)^{2} 
\exp{\left( - z^{\frac{1}{a}} \right)} dz \allowdisplaybreaks \\
&= k_{2}^{2} b \beta \int_{c / b}^{+\infty} (- b z + c)^{2} 
\exp{\left( - z^{\frac{1}{a}} \right)} dz \\
&\quad + k_{1}^{2} b \beta \int_{0}^{c / b} (- b z + c)^{2} 
\exp{\left( - z^{\frac{1}{a}} \right)} dz
+ k_{1}^{2} b \beta \int_{0}^{+\infty} (b z + c)^{2} 
\exp{\left( - z^{\frac{1}{a}} \right)} dz \allowdisplaybreaks \\
&= (k_{1}^{2} + k_{2}^{2}) b^{3} \beta 
\int_{0}^{+\infty} z^{2} \exp{\left( - z^{\frac{1}{a}} \right)} dz  
+ (k_{1}^{2} - k_{2}^{2}) b^{3} \beta 
\int_{0}^{c / b} z^{2} \exp{\left( - z^{\frac{1}{a}} \right)} dz \\
&\quad + 2 (k_{1}^{2} - k_{2}^{2}) b^{2} c \beta 
\int_{c / b}^{+\infty} z \exp{\left( - z^{\frac{1}{a}} \right)} dz \\
&\quad + (k_{1}^{2} + k_{2}^{2}) b c^{2} \beta 
\int_{0}^{+\infty} \exp{\left( - z^{\frac{1}{a}} \right)} dz 
+ (k_{1}^{2} - k_{2}^{2}) b c^{2} \beta 
\int_{0}^{c / b} \exp{\left( - z^{\frac{1}{a}} \right)} dz.  
\end{align*}
When $c < 0$,  we have 
\begin{align*}
\Ex[\Pe(Z + c)^{2}] 
&= k_{2}^{2} b \beta 
\int_{- \infty}^{0} (b z + c)^{2} \exp{\left( - (- z)^{\frac{1}{a}} \right)} dz 
+ k_{2}^{2} b \beta 
\int_{0}^{- c / b} (b z + c)^{2} \exp{\left( - z^{\frac{1}{a}} \right)} dz \\
&\quad + k_{1}^{2} b \beta 
\int_{- c / b}^{+\infty} (b z + c)^{2} \exp{\left( - z^{\frac{1}{a}} \right)} dz 
\allowdisplaybreaks \\
&= k_{2}^{2} b \beta 
\int_{0}^{+\infty} (- b z + c)^{2} \exp{\left( - z^{\frac{1}{a}} \right)} dz 
+ k_{2}^{2} b \beta 
\int_{0}^{- c / b} (b z + c)^{2} \exp{\left( - z^{\frac{1}{a}} \right)} dz \\
&\quad + k_{1}^{2} b \beta 
\int_{- c / b}^{+\infty} (b z + c)^{2} \exp{\left( - z^{\frac{1}{a}} \right)} dz 
\allowdisplaybreaks \\
&= (k_{1}^{2} + k_{2}^{2}) b^{3} \beta 
\int_{0}^{+\infty} z^{2} \exp{\left( - z^{\frac{1}{a}} \right)} dz  
- (k_{1}^{2} - k_{2}^{2}) b^{3} \beta 
\int_{0}^{- c / b} z^{2} \exp{\left( - z^{\frac{1}{a}} \right)} dz \\
&\quad + 2 (k_{1}^{2} - k_{2}^{2}) b^{2} c \beta 
\int_{- c / b}^{+\infty} z \exp{\left( - z^{\frac{1}{a}} \right)} dz \\
&\quad + (k_{1}^{2} + k_{2}^{2}) b c^{2} \beta 
\int_{0}^{+\infty} \exp{\left( - z^{\frac{1}{a}} \right)} dz 
- (k_{1}^{2} - k_{2}^{2}) b c^{2} \beta 
\int_{0}^{- c / b} \exp{\left( - z^{\frac{1}{a}} \right)} dz. 
\end{align*}
From the above, for any $c \in \mathbb{R}$, we have 
\begin{align*}
\Ex[\Pe(Z + c)^{2}] 
&= (k_{1}^{2} + k_{2}^{2}) b^{3} \beta 
\int_{0}^{+\infty} z^{2} \exp{\left( - z^{\frac{1}{a}} \right)} dz \\ 
&\quad + \Sgn(c) (k_{1}^{2} - k_{2}^{2}) b^{3} \beta 
\int_{0}^{\lvert c / b \rvert} z^{2} \exp{\left( - z^{\frac{1}{a}} \right)} dz \\
&\quad + 2 (k_{1}^{2} - k_{2}^{2}) b^{2} c \beta 
\int_{\lvert c / b \rvert}^{+\infty} z \exp{\left( - z^{\frac{1}{a}} \right)} dz \\
&\quad + (k_{1}^{2} + k_{2}^{2}) b c^{2} \beta 
\int_{0}^{+\infty} \exp{\left( - z^{\frac{1}{a}} \right)} dz \\ 
&\quad + \Sgn(c) (k_{1}^{2} - k_{2}^{2}) b c^{2} \beta 
\int_{0}^{\lvert c / b \rvert} \exp{\left( - z^{\frac{1}{a}} \right)} dz. 
\end{align*}
Now set $t := z^{\frac{1}{a}}$ to get 
\begin{align*}
\Ex[\Pe(Z + c)^{2}] 
&= (k_{1}^{2} + k_{2}^{2}) a b^{3} \beta \int_{0}^{+\infty} t^{3a - 1} e^{-t} dt 
+ \Sgn(c) (k_{1}^{2} - k_{2}^{2}) a b^{3} \beta \int_{0}^{c'} t^{3a - 1} e^{-t} dt \\
&\quad + 2(k_{1}^{2} - k_{2}^{2}) a b^{2} c \beta \int_{c'}^{+\infty} t^{2a - 1} e^{-t} dt \\
&\quad + (k_{1}^{2} + k_{2}^{2}) a b c^{2} \beta \int_{0}^{+\infty} t^{a - 1} e^{-t} dt 
+ \Sgn(c) (k_{1}^{2} - k_{2}^{2}) a b c^{2} \beta \int_{0}^{c'} t^{a - 1} e^{-t} dt 
\allowdisplaybreaks \\
&= (k_{1}^{2} + k_{2}^{2}) a b^{3} \beta \G(3a) 
+ \Sgn(c) (k_{1}^{2} - k_{2}^{2}) a b^{3} \beta \g(3a, c') \\
&\quad + 2(k_{1}^{2} - k_{2}^{2}) a b^{2} c \beta \G(2a, c') \\
&\quad + (k_{1}^{2} + k_{2}^{2}) a b c^{2} \beta \G(a) 
+ \Sgn(c) (k_{1}^{2} - k_{2}^{2}) a b c^{2} \beta \g(a, c'), 
\end{align*}
where $c' := \lvert c / b \rvert^{\frac{1}{a}}$. 
Therefore, for any $c \in \mathbb{R}$, we have 
\begin{align*}
\Ex[\Pe(Z + c)^{2}] 
&= \frac{(k_{1}^{2} + k_{2}^{2}) c^{2}}{2} 
+ \Sgn(c) \frac{(k_{1}^{2} - k_{2}^{2}) c^{2}}{2 \G(a)} 
\g\left(a, \left\lvert \frac{c}{b} \right\rvert^{\frac{1}{a}} \right) 
+ \frac{(k_{1}^{2} - k_{2}^{2}) b c}{\G(a)} 
\G\left(2a, \left\lvert \frac{c}{b} \right\rvert^{\frac{1}{a}} \right) \\
&\quad + \frac{(k_{1}^{2} + k_{2}^{2}) b^{2} \G(3a)}{2 \G(a)} 
+ \Sgn(c) \frac{(k_{1}^{2} - k_{2}^{2}) b^{2}}{2 \G(a)} 
\g\left(3a, \left\lvert \frac{c}{b} \right\rvert^{\frac{1}{a}} \right).   
\end{align*}
Also, from $(\ref{E[Pe(Z+c)]-2})$, we have 
\begin{align*}
\Ex[\Pe(Z + c)]^{2} 
&= \frac{(k_{1} - k_{2})^{2} c^{2}}{4} 
+ \frac{(k_{1}^{2} - k_{2}^{2}) c \lvert c \rvert}{2 \G(a)}
\g\left(a, \left\lvert \frac{c}{b} \right\rvert^{\frac{1}{a}} \right) 
+ \frac{(k_{1}^{2} - k_{2}^{2}) b c}{2 \G(a)}
\G\left(2a, \left\lvert \frac{c}{b} \right\rvert^{\frac{1}{a}} \right) \\
&\quad + \frac{(k_{1} + k_{2})^{2} b \lvert c \rvert}{2 \G(a)^{2}}
\g\left(a, \left\lvert \frac{c}{b} \right\rvert^{\frac{1}{a}} \right)
\G\left(2a, \left\lvert \frac{c}{b} \right\rvert^{\frac{1}{a}} \right) \\
&\quad + \frac{(k_{1} + k_{2})^{2} c^{2}}{4 \G(a)^{2}}
\g\left(a, \left\lvert \frac{c}{b} \right\rvert^{\frac{1}{a}} \right)^{2} 
+ \frac{(k_{1} + k_{2})^{2} b^{2}}{4 \G(a)^{2}}
\G\left(2a, \left\lvert \frac{c}{b} \right\rvert^{\frac{1}{a}} \right)^{2}. 
\end{align*}
Therefore, for any $c \in \mathbb{R}$, we have 
\begin{align*}
\Va[\Pe(Z + c)] 
&= \Ex[\Pe(Z + c)^{2}] - \Ex[\Pe(Z + c)]^{2} \\
&= \frac{(k_{1} + k_{2})^{2} c^{2}}{4} 
+ \frac{(k_{1}^{2} - k_{2}^{2}) b c}{2 \G(a)} 
\G\left(2a, \left\lvert \frac{c}{b} \right\rvert^{\frac{1}{a}} \right) \nonumber \\
&\quad - \frac{(k_{1} + k_{2})^{2} b \lvert c \rvert}{2 \G(a)^{2}}
\g\left(a, \left\lvert \frac{c}{b} \right\rvert^{\frac{1}{a}} \right)
\G\left(2a, \left\lvert \frac{c}{b} \right\rvert^{\frac{1}{a}} \right) \nonumber \\
&\quad - \frac{(k_{1} + k_{2})^{2} c^{2}}{4 \G(a)^{2}}
\g\left(a, \left\lvert \frac{c}{b} \right\rvert^{\frac{1}{a}} \right)^{2} 
- \frac{(k_{1} + k_{2})^{2} b^{2}}{4 \G(a)^{2}}
\G\left(2a, \left\lvert \frac{c}{b} \right\rvert^{\frac{1}{a}} \right)^{2} \nonumber \\
&\quad + \frac{(k_{1}^{2} + k_{2}^{2}) b^{2} \G(3a)}{2 \G(a)}  
+ \Sgn(c) \frac{(k_{1}^{2} - k_{2}^{2}) b^{2}}{2 \G(a)} 
\g\left(3a, \left\lvert \frac{c}{b} \right\rvert^{\frac{1}{a}} \right). \nonumber
\end{align*}

\clearpage

\bibliography{reference}

\begin{thebibliography}{10}

\bibitem{doi:10.1111/j.1751-5823.1998.tb00406.x}
John Aldrich.
\newblock Doing least squares: Perspectives from gauss and yule.
\newblock {\em International Statistical Review}, 66(1):61--81, 1998.

\bibitem{andrews_askey_roy_1999}
George~E. Andrews, Richard Askey, and Ranjan Roy.
\newblock {\em Special Functions}.
\newblock Encyclopedia of Mathematics and its Applications. Cambridge
  University Press, 1999.

\bibitem{10.2307/2336317}
Jens Breckling and Ray Chambers.
\newblock M-quantiles.
\newblock {\em Biometrika}, 75(4):761--771, 1988.

\bibitem{Dytso2018}
Alex Dytso, Ronit Bustin, H.~Vincent Poor, and Shlomo Shamai.
\newblock Analytical properties of generalized gaussian distributions.
\newblock {\em Journal of Statistical Distributions and Applications}, 5(1):6,
  Dec 2018.

\bibitem{10.2307/24303995}
B.~Efron.
\newblock Regression percentiles using asymmetric squared error loss.
\newblock {\em Statistica Sinica}, 1(1):93--125, 1991.

\bibitem{10.1145/3236009}
Riccardo Guidotti, Anna Monreale, Franco Turini, Dino Pedreschi, and Fosca
  Giannotti.
\newblock A survey of methods for explaining black box models.
\newblock {\em ACM Computing Surveys}, 51, 02 2018.

\bibitem{10.2307/1913643}
Roger Koenker and Gilbert Bassett.
\newblock Regression quantiles.
\newblock {\em Econometrica}, 46(1):33--50, 1978.

\bibitem{legendre1805nouvelles}
A.M. Legendre.
\newblock {\em Nouvelles m{\'e}thodes pour la d{\'e}termination des orbites des
  com{\`e}tes}.
\newblock Nineteenth Century Collections Online (NCCO): Science, Technology,
  and Medicine: 1780-1925. F. Didot, 1805.

\bibitem{doi:10.1080/02664760500079464}
Saralees Nadarajah.
\newblock A generalized normal distribution.
\newblock {\em Journal of Applied Statistics}, 32(7):685--694, 2005.

\bibitem{stigler1981}
Stephen~M. Stigler.
\newblock Gauss and the invention of least squares.
\newblock {\em Ann. Statist.}, 9(3):465--474, 05 1981.

\bibitem{Sub23}
Th. {Subbotin}.
\newblock {On the law of frequency of error.}
\newblock {\em {Recueil Math\'ematique}}, 31:296--301, 1923.

\bibitem{doi:10.1142/0653}
Z~X Wang and D~R Guo.
\newblock {\em Special Functions}.
\newblock WORLD SCIENTIFIC, 1989.

\bibitem{Yamaguchi2018}
Naoya Yamaguchi, Maiya Hori, and Yoshinari Ideguchi.
\newblock Minimising the expectation value of the procurement cost in
  electricity markets based on the prediction error of energy consumption.
\newblock {\em Pacific Journal of Mathematics for Industry}, 10(1):4, Aug 2018.

\bibitem{10.2307/2289234}
Arnold Zellner.
\newblock Bayesian estimation and prediction using asymmetric loss functions.
\newblock {\em Journal of the American Statistical Association},
  81(394):446--451, 1986.

\end{thebibliography}
\bibliographystyle{plain}

\medskip
\begin{flushleft}
Naoya Yamaguchi\\
Office for Establishment of an Information-related School\\
Nagasaki University\\
1-14 Bunkyo, Nagasaki City 852-8521 \\
Japan\\
yamaguchi@nagasaki-u.ac.jp
\end{flushleft}

\vspace{0.3cm}

\begin{flushleft}
Yuka Yamaguchi\\
Office for Establishment of an Information-related School\\
Nagasaki University\\
1-14 Bunkyo, Nagasaki City 852-8521 \\
Japan\\
yamaguchiyuka@nagasaki-u.ac.jp
\end{flushleft}

\vspace{0.3cm}

\begin{flushleft}
Ryuei Nishi\\
Office for Establishment of an Information-related School\\
Nagasaki University\\
1-14 Bunkyo, Nagasaki City 852-8521 \\
Japan\\
nishii.ryuei@nagasaki-u.ac.jp
\end{flushleft}

%\begin{thebibliography}{99}
%\bibitem{AAR}
%Andrews, George E., Richard Askey, and Ranjan Roy, {\itshape Special functions}, Vol. 71, Cambridge university press (2000).
%\end{thebibliography}

\end{document}